\pdfoutput=1
\documentclass[9pt,twoside,epsf]{article}
\usepackage{tikz}
\usetikzlibrary{calc,arrows.meta,angles,quotes,intersections}
\usepackage{graphicx,epstopdf} % <- Preamble
\usepackage[caption=false]{subfig} % <- Preamble
\usepackage{stmaryrd}
\usepackage{amsmath,amsthm,latexsym,amssymb,amsfonts,amsbsy}
\usepackage{bm}
\usepackage{graphicx}
\usepackage{cases}
\usepackage{multirow}
\usepackage{booktabs}
\usepackage{diagbox}
\usepackage{float}
\usepackage{needspace}
\usepackage[hidelinks]{hyperref}
\usepackage[capitalise]{cleveref}
\crefname{equation}{}{}
\Crefname{equation}{}{}
\usepackage{algorithm}
\usepackage{algorithmic}

\usepackage{lineno}
\usepackage{xcolor}
\newfont{\bb}{msbm10}
\def\Bbb#1{\mbox{\bb #1}}

\def\T{\top}

\def\diag{{\rm diag}}

\def\E{\mathbb{E}}
\def\rank{{\rm rank}}
\def\range{{\rm range}}

\def\zspace{{\rm null}}

\def\pr{{\rm Pr}}

\def\cov{{\rm Cov}}

\usepackage{colortbl}
\usepackage{ragged2e}
\usepackage[margin=2em,labelsep=space,skip=0.5em,font=normalfont]{caption}
\floatname{algorithm}{\color{black} Algorithm}%{\color{blue} Algorithm}
\makeatletter
\renewcommand{\fnum@algorithm}{\textbf{Algorithm~\thealgorithm}} % 加粗算法的标题
\makeatother

\numberwithin{algorithm}{section}


\newtheorem{remark}{Remark}[section]
\newtheorem{theorem}{Theorem}[section]

\newtheorem{lemma}{Lemma}[section]

\newcommand{\reals}{\makebox{{\Bbb R}}}

\newcommand{\bbb}[1]{\text{\bf #1}}
\newcommand{\txt}[1]{\text{\rm #1}}
\newcommand{\txttiny}[1]{\text{\tiny\rm #1}}

\makeatletter % '@' is now a normal "letter" for TeX

\@addtoreset{equation}{section}
\makeatother % '@' is restored as a "non-letter" character for TeX

\begin{document}
%\pagewiselinenumbers
\cleardoublepage
\pagestyle{myheadings}

\bibliographystyle{plain}

\title{\bf{On randomized batch-sampling coordinate descent methods for solving linear least-squares problems}
\thanks{First author: Dong-Yue Xie (xiedongyue@nuaa.edu.cn); Corresponding author: Xi Yang (yangxi@nuaa.edu.cn).}}

\author{Dong-Yue Xie\thanks{School of Mathematics, Nanjing University of Aeronautics and Astronautics, Nanjing 211106, China.}, Xi Yang\footnotemark[2]}
\maketitle

\begin{abstract}
This paper refines randomized block coordinate descent by introducing a coordinate-level batch-sampling distribution characterization, yielding a fine-grained randomized batch-sampling coordinate descent framework, which applies to any block coordinate descent method under an arbitrary static stochastic sampling rule for solving large-scale linear least-squares problems. Based on this framework, we derive new linear convergence rate bounds that are invariant to the column scaling of data matrices. Theoretical analysis and numerical experiments validate that the new bounds are significantly sharper than existing results and better aligned with the practical convergence rate. Furthermore, we proposed a learning-based sampling strategy by parameterizing the batch-sampling distribution and optimizing its parameters through gradient-based training. Preliminary numerical results on computed tomography (CT) reconstruction problems demonstrate that the resulting method converges significantly faster than variants using other fixed sampling rules, and outperforms tested adaptive coordinate descent methods and full gradient descent in computational efficiency, confirming that learned sampling distributions can effectively capture problem-specific structure to accelerate convergence within the framework.
\bigskip

\noindent{\bf Keywords:} linear least-squares problem, batch-sampling,
randomized coordinate descent, convergence rate, learning-based sampling

\noindent{\bf MSC codes:} 65F10, 65F20, 65K05, 68W20
\end{abstract}

\section{Introduction}
\label{sec-introduction}
Consider the large-scale linear least-squares problem of the form
\begin{equation}
	\label{least-squares-prob}
	% \nonumber % Remove numbering (before each equation)
	\min_{x\in \mathbb{R}^{n}}\|Ax-b\|^{2}_{2},\quad\mbox{with $A\in\reals^{m\times n}$, $b\in\reals^m$}.
\end{equation}
The matrix $A$ is real and of full column rank, which appears in many real-world applications.
Coordinate descent (CD) method \cite{Wright2015CD} is a classical iterative algorithm for large-scale optimization, whose development can be traced to the foundational work in mathematical optimization. At each step, the CD method updates the current iterate by selecting one coordinate while keeping the remaining components fixed, and then approximately minimizing the objective function with respect to the chosen coordinate. The one-dimensional subproblem can be solved much more efficiently than the original high-dimensional problem. Because of this computational advantage, the CD method has found broad applications in optimization and related fields, such as tomography \cite{Bouman96}, tensor factorization \cite{Razaviyayn13}, compressed sensing \cite{LiOsher09,Sardy00}, image reconstruction \cite{Bjorck79,Elfving17}, and machine learning \cite{Hsieh08,Chang08}.

In the optimization literature, substantial theoretical work has been devoted to the CD method. Research in this direction has established convergence guarantees and iteration complexity bounds for various convex optimization problems, and further developed accelerated variants that take advantage of problem structure. For detailed references, see, e.g., \cite{Wright2015CD,Beck13,LuXiao2015RBCD,Nesterov2012CD,Richtarik12,Richtarik16,Necoara17,Nesterov17} and the references therein. These contributions demonstrate the effectiveness of coordinate descent as a broadly applicable optimization technique. As a special class of optimization problems, using the CD method to solve the least-squares problems is mathematically equivalent to applying the classical Gauss-Seidel iteration method \cite{Ruhe1983GS} to the normal equation
\begin{equation*}
	A^{\T}Ax=A^{\T}b.
\end{equation*}
Leventhal and Lewis \cite{Leventhal10} proposed a randomized coordinate descent (RCD) method for least-squares problems where each coordinate is selected according to a prescribed probability distribution. They also proved that the method attains linear convergence in expectation, where the rate is determined by natural linear-algebraic measures of \cref{least-squares-prob} (e.g., the scaled condition number of $A$, defined by $\kappa(A)=\|A\|_F\|A^\dag\|_2$). Following this seminal work, numerous studies on RCD-type methods have been carried out, including the development of new variants of the RCD method, such as extended variants \cite{Ma15} and new methods based on different stochastic sampling rules \cite{BaiWu19NLAA,ZhangGuo2020RGRCD}, as well as further analyses of convergence theory \cite{Du2019Tight,BaiWangWu2021RGS}.

In order to achieve faster convergence and to utilize parallel computation more effectively, block versions of coordinate descent (BCD) have been proposed for solving least-squares problems. In these methods, a block of coordinates is updated simultaneously at each iteration, rather than a single coordinate, which can reduce the number of iterations and take advantage of modern parallel computing architectures. Let $\tau \subseteq \{1,\ldots,n\}$ be a set of column indices with $|\tau|=q$, and denote by $A_{\tau}$ the submatrix of $A$ formed by the columns of $A$ with indices in $\tau$. Let $E_{\tau}$ be the coordinate sampling matrix whose columns are the standard basis vectors $e_i\in\reals^n$ for $i\in\tau$. Then the update of BCD at the $(k+1)$th iteration is given by
\begin{equation}
	\label{RBCD-method}
	x^{(k+1)} = x^{(k)} + E_{\tau^{(i_k)}} A_{\tau^{(i_k)}}^{\dag}(b-Ax^{(k)}),
\end{equation}
where $A_{\tau^{(i_k)}}^{\dag}$ denotes the Moore-Penrose pseudoinverse of $A_{\tau^{(i_k)}}$. Following the work of Needell and Tropp \cite{NeedellTropp2014LAA}, Needell et al. \cite{NeedellLAA15, WuNeedell2018RBGS} first introduced the randomized block coordinate descent (RBCD) method and established its expected linear convergence. In recent years, inspired by the RBCD method, several new variants of BCD have been proposed which achieve faster convergence rates and improved computational efficiency (see, e.g., \cite{LiZhang2020GBGS,LiuJiangGu2021MRBGS,NiuZheng2024ABC,TanDengQiu2024ADBCD,DuSun2021DSBGS}). In addition, some of the block coordinate descent variants mentioned above can be viewed as special cases of the following sketch-and-project (SAP, first proposed by Gower et al. in \cite{gower2015randomized}) framework
\begin{align}
\label{SAP-framework}
	x^{(k+1)}
	={}&x^{(k)}
	+B^{-1}A^{\top}\mathcal S_k
	\left(
	\mathcal S_k^{\top}AB^{-1}A^{\top}\mathcal S_k
	\right)^{\dagger}
	\mathcal S_k^{\top}
	\left(b-Ax^{(k)}\right),
\end{align}
where $B\in\reals^{n\times n}$ is symmetric positive definite, and $\mathcal S_k \in\reals^{m\times q}$ is the sketching matrix. For example, when setting $\mathcal S_k = A_{\tau^{(i_k)}}$ and $B = A^{\T}A$, the SAP update \cref{SAP-framework} reduces to the BCD update \cref{RBCD-method}. Furthermore, the general convergence theory for SAP is established in \cite{gower2015randomized,richtarik2020,gower2021adaptive}.

As a counterpart to the componentwise updates of coordinate descent, Kaczmarz-type (row-action) methods perform at each step a projection correction onto the solution set of a single equation or a block of equations, forming another mainstream paradigm for large-scale linear least-squares problems and noisy linear systems. For linear least-squares, the extended (Zouzias et al. \cite{Zouzias2013}) and augmented (Bai et al.~\cite{BaiWu2021}; Li et al.~\cite{LiWu2025}; He et al.~\cite{HeWangLi}) Kaczmarz iterations drive the iterates to the least-squares solution by alternating row- and column-space projections. When the right-hand side or the coefficient matrix is perturbed, the effect of noise on the iterates and its control become central. Bergou et al.~\cite{Bergou2024} characterized the convergence behavior of randomized Kaczmarz for doubly noisy linear systems. Xie et al.~\cite{Xie2023} introduced adaptive stepsizes to prevent stagnation caused by noise. Alderman et al.~\cite{Alderman2024} employed geometrically smoothed momentum to suppress the propagation and amplification of noise along the iteration. The tail-averaging technique of Epperly et al.~\cite{Epperly2026} yields explicit bounds on variance decay in the noisy regime. Furthermore, combining Tikhonov-regularized projections with adaptive momentum delivers convergence rates that go beyond what Krylov subspace methods can attain for noisy, approximately low-rank systems (Derezi\'nski et al.~\cite{Derezinski2025}). On the theoretical side, Hansen et al.~\cite{HansenHochstenbach} elucidated, from a spectral viewpoint, the mechanism by which data noise propagates into the iterate, thereby providing a basis for early-stopping criteria. Overall, row-action methods acting on the solution sets of subsets of equations and coordinate descent acting on individual or subsets of variable components constitute two parallel and complementary routes to solve linear least-squares problems and noisy linear systems; the present work proceeds along the latter. 

\textbf{Motivation.} Notably, methods employing greedy or adaptive sampling rules, which utilize current iteration information to select working blocks, significantly increase the per-iteration cost. This is due to the fact that such methods require substantial additional computations to determine the working block and often lead to non-contiguous memory accesses. As a result, the practical efficiency of these algorithms depends not only on the convergence rate but also on the computational cost of determining the working block at each iteration. Therefore, we will focus on optimizing conventional RBCD methods rather than their greedy or adaptive variants. The motivation for present work is that existing block sampling rules for RBCD-type methods are usually described through relatively coarse probabilistic models, which do not explicitly characterize the coordinate-level joint sampling structure. For example, one typically first constructs a family of candidate coordinate blocks according to some prescribed rule and then defines a probability distribution over these blocks for sampling \cite{gower2015randomized,richtarik2020,gower2021adaptive}. Although such descriptions are effective for particular sampling mechanisms, they are not sufficiently flexible to characterize general coordinate-batch sampling rules. Moreover, the most practical and frequently used existing sampling rules (e.g., uniform sampling or sampling proportional to column scalings) are generally not optimized, which may limit the performance that randomized methods can achieve. In addition, existing convergence analyses for RBCD-type methods are often derived based on worst-case analysis. As a result, the theoretical bounds on the convergence rate tend to be overly conservative and may not accurately reflect the algorithm's empirical performance.

\textbf{Main contributions.} Motivated by these considerations, the main contributions of this work are summarized as follows.
\Needspace{8\baselineskip}
\begin{enumerate}
\renewcommand{\labelenumi}{\textbf{(\roman{enumi})}}
\setlength{\topsep}{3pt}
\setlength{\itemsep}{5pt}
\setlength{\parsep}{0pt}
\item \textbf{A fine-grained framework based on general batch-sampling distributions \(\bbb{P}\).} We reformulate a fine-grained randomized batch-sampling coordinate descent (RBSCD) framework for least-squares problems under arbitrary fixed stochastic sampling rules. The coordinate batch selected at each iteration is modeled as a random vector whose stochastic sampling rule is fully described by its joint distribution \(\bbb{P}\) (see \cref{sec-def-RBSCD}). This fine-grained definition of the sampling distribution provides the framework with a high degree of generality and flexibility and enables a unified convergence analysis of RBCD-type methods. The resulting convergence factor depends explicitly on \(\bbb{P}\), thereby suggesting a new approach to developing more efficient coordinate descent methods by optimizing the sampling distribution. Such optimization can induce more effective update strategies and yield randomized methods with the potential to approach their intrinsic performance limits. Notably, since \(\bbb{P}\) is known during the iterations and its use does not require current-iterate information, the framework retains the low per-iteration sampling cost of conventional RBCD methods.

\item \textbf{A scale-invariant and sharper convergence analysis based on auxiliary operators \(S\).} We introduce an arbitrary invertible diagonal scaling operator \(S\) into the convergence analysis. This operator is used only to derive and sharpen the theoretical bound and is neither formed nor applied in the RBSCD iteration. Its introduction yields convergence bounds that are invariant with respect to the column scaling of the data matrix. Moreover, for a fixed sampling distribution \(\bbb{P}\), minimizing the convergence factor \(\Phi(\bbb{P},S)\) (see \cref{convRBSCD}) over \(S\) can further tighten the theoretical estimate without altering the underlying algorithm. A preliminary numerical experiment in \cref{section 5.1} shows that optimizing \(S\) can substantially reduce the gap between the theoretical bound and the empirical convergence rate. To the best of our knowledge, such a data-scaling-free convergence analysis based on an optimizable auxiliary scaling operator has not previously been established for RBCD methods. 

\item \textbf{A sharper bound exploiting update-induced subspace orthogonality.} The preceding RBSCD update implies that the current error satisfies an additional orthogonality condition relative to the previously selected block. By retaining this iteration-dependent information, we use \(W_{\tau}\) to represent the resulting restricted-subspace and evaluate the relevant eigenvalue over this subspace (see \cref{lem7} for details). This yields a sharper convergence factor bound involving
\(
\lambda_{\min}\!\left(
W_{\tau}^{\top}AM^{2}A^{\top}W_{\tau}
\right),
\) instead of an eigenvalue bound over the entire space. The resulting worst-case convergence rate bound is sharper than the classical rate bounds of Needell et al. \cite{NeedellLAA15,WuNeedell2018RBGS}, as demonstrated both theoretically and numerically. 

\item \textbf{A non-worst-case convergence estimate.} Building on the RBSCD framework and combining the analysis based on \(S\) with the restricted-subspace analysis represented by \(W_{\tau}\), we further employ concentration inequalities and the law of large numbers to derive a tighter non-worst-case convergence estimate. Instead of relying on the uniform worst-case quantity \(\min_{\tau\sim\bbb{P}}\xi_{\tau}\) (see \cref{lem7} and \cref{convRBSCD}), the resulting convergence factor is expressed in terms of \(\mathbb{E}(\xi_{\tau})\) (see \cref{convRBSCD_sharper}), which accounts for the effectiveness of different batches according to their sampling probabilities and therefore more faithfully characterizes the actual convergence behavior of the randomized iteration. Theoretical results and numerical experiments show that the proposed bounds are sharper than existing bounds and better reflect the empirical convergence behavior. 

\item \textbf{Reuse strategy of learned optimal distributions and efficient implementation.} As discussed in (i), the explicit dependence of the convergence factor on \(\bbb{P}\) provides a principled optimization objective for constructing a more effective sampling distribution. Accordingly, We optimize the distribution parameters through gradient-based training. We also develop an efficient implementation of RBSCD that reduces the per-iteration computational cost through precomputation and efficient update strategies (see \cref{eff-RBSCD}). Numerical results on CT reconstruction problems show that the learned optimal distribution reused RBSCD method converges faster than variants using standard fixed sampling distributions and achieves higher computational efficiency than the other tested coordinate-descent methods. These results demonstrate that the learned optimal distributions reuse strategy provides a feasible and effective approach to improving the performance of CD-type methods. 
    
\end{enumerate}

The organization of the paper is as follows. In \cref{sec-def-RBSCD}, we present the framework of the RBSCD method and describe an efficient implementation of it. In \cref{sec-lemmas}, we establish several auxiliary lemmas that are instrumental for the convergence analysis. In \cref{RBSCD-analyse}, we prove the expected linear convergence of the RBSCD method. Numerical experiments are reported in \cref{sec-experiments}. \Cref{sec-conclusions} ends the paper with concluding remarks and directions for future research.

At the end of this section, we introduce the notation adopted in this paper. $\mathbb{DR}^{n\times n}$ is denoted by the set of all invertible diagonal matrices in $\reals^{n\times n}$. For any matrix $G \in \reals^{m \times n}$, $G^{\top}$, $G^{\dag}$, $\|G\|_2$, $\|G\|_F$ and $\range(G)$ are used to represent the transpose, the Moore-Penrose pseudoinverse, the spectral norm, the Frobenius norm and the column space of $G$, respectively. For a general matrix $G$, we denote its positive singular values by $\sigma_{i}(G)$, and in particular, $\sigma_{\min}(G)$ and $\sigma_{\max}(G)$ are the smallest and the largest positive singular values of $G$. For any symmetric positive semi-definite matrix $G$, we define the seminorm induced by $G$ as $\|\cdot\|_{G} =\sqrt{\langle \cdot,\, G\,\cdot\rangle}$, and we denote by $\lambda_{\min}(G)$ and  $\lambda_{\max}(G)$ the smallest and the largest positive eigenvalues of $G$. For Hermitian matrices $G$ and $H$ of the same size, $G\succ H$ means that $G-H$ is positive definite. For any vector $b \in \reals^m$, we denote by $b_{\mathcal{R}(A)}$ its orthogonal projection onto $\range(A)$, and by $b_{\mathcal{R}(A)^{\perp}}=(I - AA^{\dag})b$ its orthogonal residual. We denote by $[n]$ the index set including integers $\{1,\ldots,n\}$, $[n]^q$ the $q$-fold Cartesian product of $[n]$ (i.e., the set of all ordered $q$-tuples with entries in $[n]$), $|\cdot|$ the absolute value of a real number or the number of elements of a set.
Moreover, let $\mathbb{E}_k$ be the conditional expectation with respect to the first $k$ iterations, that is,
\begin{equation*}
	% \nonumber % Remove numbering (before each equation)
	\mathbb{E}_k(\cdot)=\mathbb{E}\left(\cdot \mid \tau^{(i_0)}, \tau^{(i_1)}, \ldots, \tau^{(i_{k-1})}\right),
\end{equation*}
where $\tau^{(i_\ell)}$ ($\ell=0,1, \ldots, k-1$) collects the column indices selected at the $\ell$th iterate. Then by the law of iterated expectation, we have  $\mathbb{E}\left(\mathbb{E}_k(\cdot)\right)=\mathbb{E}(\cdot)$.

\section{The RBSCD method}\label{sec-def-RBSCD}
We adopt the batch-sampling notation and the associated joint distribution model introduced in Section~3 of \cite{XieYang2025RBSK} to describe the coordinate-batch sampling rule. Specifically, for each iteration we call a random vector $\tau = (\tau_1, \ldots, \tau_q) \in [n]^q$ as a \textsc{batch-sampling} from the set $[n]$ with batch size $q$, if the batch $\tau$ follows a joint probability distribution $\bbb{P} = (\bbb{p}_{j_1, \ldots, j_q})$, where $\pr(\tau = (j_1, \ldots, j_q)) = \bbb{p}_{j_1, \ldots, j_q} \geq 0$ and the distribution is normalized such that $\sum_{j_1, \ldots, j_q = 1}^n \bbb{p}_{j_1, \ldots, j_q} = 1$. For each random coordinate index $\tau_i$, $i=1, \ldots, q$, we define the marginal distribution
\begin{equation}
	\label{marDistrib}
	p_{ij} =\pr(\tau_i=j) =\ \sum_{j_1,\ldots,j_{i-1},j_{i+1},\ldots,j_q=1}^{n}\bbb{p}_{j_1,\ldots,j_{i-1},j_i=j,j_{i+1},\ldots,j_q},\quad j\in[n],
\end{equation}
the corresponding diagonal matrices
\begin{equation}
	\label{marDistribMatrix}
	P_i=\diag\left(p_{i1},p_{i2},\ldots,p_{in}\right),%\quad\mbox{ for $i=1,\ldots,q,$}
\end{equation}
and a well-defined batch-sampling should lead to the following property
\begin{equation}
	\label{marDistribMatrixProperty}
	\sum_{i=1}^{q}P_i \succ 0.
\end{equation}

The joint distribution $\bbb{P}$ can be interpreted as a $q$-way probability tensor over $[n]^q$. It serves as the sampling parameter that specifies how the batch $\tau$ is generated. Consequently, different choices of $\bbb{P}$ induce different stochastic sampling rules and thus define various randomized block coordinate descent methods.

By the definition of batch-sampling, a sampled batch $\tau$ may contain repeated indices, we denote by $\tau_{\txttiny{E}}$ the effective set of $\tau$, i.e., the set of distinct indices appearing in $\tau$. For a prescribed effective set $\tau_{\txttiny{E}}$, define its preimage by
\begin{equation*}
	\tau^{-1}(\tau_{\txttiny{E}})
	= \left\{\widehat{\tau}\in[n]^q \ \big| \ \mbox{the effective set of $\widehat{\tau}$ is $\tau_{\txttiny{E}}$} \right\},
\end{equation*}
and define the induced probability as
\begin{equation*}
	\pr(\tau_{\txttiny{E}})
	=\pr\!\left(\tau\in \tau^{-1}(\tau_{\txttiny{E}})\right)
	=\sum_{\tau\in\tau^{-1}(\tau_{\txttiny{E}})}\pr(\tau),
\end{equation*}
which represents the probability that a draw of the batch-sampling $\tau$ leads to the effective set $\tau_{\txttiny{E}}$. For concrete examples of sampling rules induced by the batch-sampling $\tau\sim\bbb{P}$, as well as further details on batch-sampling, the reader can refer to \cite{XieYang2025RBSK}.

To describe the RBSCD method, for any matrix $G$ and any batch $\tau\in[n]^q$, we denote by $G_{\tau}$ the batch-sampling matrix formed by the columns of $G$ indexed by $\tau$. Moreover, we consider a finite collection of candidate coordinate batches $\{\tau^{(i)}\}_{i=1}^{s}\subseteq[n]^q$, where $s$ denotes the number of available batches. At the $k$th iteration, we draw a batch $\tau^{(i_k)}$ according to the prescribed distribution $\bbb{P}$ and apply the update associated with $\tau^{(i_k)}$. The RBSCD method is presented as \cref{alg-RBSCD}.
\begin{algorithm}
	\caption{The RBSCD method}
	\label{alg-RBSCD}
	\begin{algorithmic}[1]
		\STATE{\textbf{Input}: $A$, $b$, $x^{(0)}$, stopping criterion, maximal iteration count IterMax;}
		\STATE{\textbf{Output}: approximate solution x;}
		\STATE{set $r^{(0)} = b - A x^{(0)}$;}
		\FOR{$k = 0$ to IterMax-1}
		\IF{the stopping criterion holds}
		\STATE{$x=x^{(k)}$; \textbf{break};}
		\ENDIF
		\STATE{Draw a batch-sampling $\tau^{(i_k)}\sim\bbb{P}$;}
		\STATE{Update the $k$th iterate $x^{(k)}$ by}
		$x^{(k+1)} = x^{(k)} + E_{\tau^{(i_k)}} A_{\tau^{(i_k)}}^{\dag}r^{(k)};$
		\STATE{Compute the residual vector $r^{(k+1)} = r^{(k)}-A_{\tau^{(i_k)}}A_{\tau^{(i_k)}}^{\dag}r^{(k)};$}
		\ENDFOR
	\end{algorithmic}
\end{algorithm}

To further improve the computational efficiency of the RBSCD method, we present in \cref{eff-RBSCD} a computationally efficient implementation.
%which deals with the case of a moderate-size predefined family of candidate batches.
The main idea is to precompute and store the batch-dependent matrix quantities that appear in the update, and to employ an efficient update strategy that reduces the arithmetic cost per iteration. In a direct implementation of \cref{alg-RBSCD}, the dominant cost lies in computing $A_{\tau^{(i_k)}}^{\dag}r^{(k)}$, i.e., solving a least-squares subproblem with coefficient matrix $A_{\tau^{(i_k)}} \in \reals^{m \times q}$. Noting that the update of $x$ can be written equivalently as
\begin{equation*}
	x^{(k+1)}
	= x^{(k)} + \left[ E_{\tau^{(i_k)}}(A_{\tau^{(i_k)}}^{\T}A_{\tau^{(i_k)}})^{\dag} \right] (A_{\tau^{(i_k)}}^{\T}r^{(k)}),
\end{equation*}
when the number of candidate batches $s$ is not too large, we can precompute and store $E_{\tau^{(i)}}(A_{\tau^{(i)}}^{\T}A_{\tau^{(i)}})^{\dag}$ for $i=1,2,\ldots,s$, thereby avoiding per-iteration pseudoinverse computations. To further accelerate the updates, we introduce the auxiliary vector sequence
\begin{equation*}
	D_i^{k}=A_{\tau^{(i)}}^{\T}r^{(k)} \in \reals^{q}, \quad\mbox{for $i=1,\ldots,s,$}
\end{equation*}
so that the RBSCD update can be expressed as
\begin{equation*}
	x^{(k+1)}=x^{(k)} + E_{\tau^{(i_k)}}(A_{\tau^{(i_k)}}^{\T}A_{\tau^{(i_k)}})^{\dag}D_{i_k}^{k}.
\end{equation*}
Moreover, since
$r^{(k+1)}=r^{(k)}-A_{\tau^{(i_k)}}(A_{\tau^{(i_k)}}^{\T}A_{\tau^{(i_k)}})^{\dag}D_{i_k}^{k}$,
then for $i=1,\ldots,s$, we obtain the recurrence
\begin{equation*}
	D_i^{k+1}
	=A_{\tau^{(i)}}^{\T}r^{(k+1)}
	=D_i^{k}-A_{\tau^{(i)}}^{\T}A_{\tau^{(i_k)}}(A_{\tau^{(i_k)}}^{\T}A_{\tau^{(i_k)}})^{\dag}D_{i_k}^{k},
\end{equation*}
which can be implemented efficiently once the matrices $A_{\tau^{(i)}}^{\T}A_{\tau^{(j)}}(A_{\tau^{(j)}}^{\T}A_{\tau^{(j)}})^{\dag}$ are precomputed. These observations lead to \cref{eff-RBSCD}, which is algebraically equivalent to \cref{alg-RBSCD}. With the batch-dependent matrices precomputed, each iteration reduces to low-dimensional ($q$-dimensional) matrix-vector multiplications and simple vector updates.

\begin{algorithm}
	\caption{Computationally efficient RBSCD method}
	\label{eff-RBSCD}
	\begin{algorithmic}[1]
		\STATE{\textbf{Input}: $A$, $b$, $x^{(0)}$, the number of batches $s$, stopping criterion, maximal iteration count IterMax;}
		\STATE{\textbf{Output}: approximate solution $x$;}
		\STATE{Precompute and store $E_{\tau^{(i)}}(A_{\tau^{(i)}}^{\T}A_{\tau^{(i)}})^{\dag}$ for $i=1,2,\cdots,s$;}
		\STATE{Precompute and store $A_{\tau^{(i)}}^{\T}A_{\tau^{(j)}}(A_{\tau^{(j)}}^{\T}A_{\tau^{(j)}})^{\dag}$ for $i,j=1,2,\cdots,s$;}
		\STATE{Precompute and store $D_{i}^{0}=A_{\tau^{(i)}}^{\T}(b-Ax^{(0)})$ for $i=1,2,\cdots,s$;}
		\FOR{$k = 0$ to IterMax-1}
		\IF{the stopping criterion holds}
		\STATE{$x=x^{(k)}$; \textbf{break};}
		\ENDIF
		\STATE{Draw a batch-sampling $\tau^{(i_k)}\sim\bbb{P}$;}
		\STATE{Update the $k$th iterate $x^{(k)}$ by
			$x^{(k+1)} = x^{(k)} + E_{\tau^{(i_k)}}(A_{\tau^{(i_k)}}^{\T}A_{\tau^{(i_k)}})^{\dag} D_{i_k}^{k};$}
		\STATE{Update and store $D_{i}^{k+1}=D_{i}^{k}-A_{\tau^{(i)}}^{\T}A_{\tau^{(i_k)}}(A_{\tau^{(i_k)}}^{\T}A_{\tau^{(i_k)}})^{\dag}D_{i_k}^{k}$ for $i=1,2,\cdots,s$;}
		\ENDFOR
	\end{algorithmic}
\end{algorithm}
\medskip
\noindent\textbf{Complexity comparison.}
We next compare the per-iteration cost of \cref{alg-RBSCD,eff-RBSCD}. Since CD-type methods generally require numerous iterations to reach a prescribed accuracy on large-scale problems, especially for sparse data matrices, the one-time preprocessing cost can be amortized over iterations and is often negligible when the number of candidate batches $s$ is not too large (e.g., when $qs$ is not significantly larger than the number of columns $n$). Therefore, we focus on the dominant floating point operations (FLOPs) per iteration, omitting lower-order terms, and analyze memory requirements.

For a direct implementation of the RBSCD method (\cref{alg-RBSCD}), sampling $i_k$ (equivalently, drawing $\tau^{(i_k)}\sim\bbb{P}$) from a prescribed distribution can be performed with $\mathcal{O}(1)$ cost after standard preprocessing. The dominant per-iteration cost arises from computing $u =A_{\tau^{(i_k)}}^{\dag}r^{(k)}$. For instance, when $A_{\tau^{(i_k)}}$ has full column rank, one may compute $u$ through the normal equation formulation $u=(A_{\tau^{(i_k)}}^{\T}A_{\tau^{(i_k)}})^{-1}A_{\tau^{(i_k)}}^{\T}r^{(k)}$, where the $q\times q$ Gram matrix $A_{\tau^{(i_k)}}^{\T}A_{\tau^{(i_k)}}$ is factorized via a Cholesky decomposition. This computation typically dominates the iteration and costs $mq^2+\frac{1}{3}q^3$ FLOPs (see Section 5.3 in \cite{golub2013matrix}). Given $u$, updating $x^{(k+1)}=x^{(k)}+E_{\tau^{(i_k)}}u$ requires $\mathcal{O}(q)$ operations, while updating the residual $r^{(k+1)}=r^{(k)}-A_{\tau^{(i_k)}}u$ costs $2mq$ FLOPs. Alternative direct solvers, such as Householder QR, have the same leading-order cost in terms of $m$ and $q$. Meanwhile, the cost of iterative methods, such as Krylov-based least-squares solvers depends on the number of inner iterations and the target accuracy.

For the efficient RBSCD method (\cref{eff-RBSCD}), since the batch-dependent matrices are precomputed and stored, each iteration reduces to updating $x^{(k+1)}$ via one $n\times q$ matrix-vector multiplication, together with updating the $s$ auxiliary vectors $\{D_i^{k}\}_{i=1}^{s}$, where each update involves one $q\times q$ matrix-vector multiplication. Consequently, the per-iteration cost is $2qn+2q^2s$ FLOPs, which is substantially lower than the $mq^2+\frac{1}{3}q^3+2mq$ FLOPs required by \cref{alg-RBSCD}, especially when $m$ is large and $s$ is moderate.

In terms of memory, \cref{alg-RBSCD} stores the current iterate $x^{(k)}\in\reals^{n}$ and the residual vector $r^{(k)}\in\reals^{m}$. In addition, to improve data access and avoid repeated non-contiguous memory reads during the iterations, one may store the preselected submatrices $\{A_{\tau^{(i)}}\}_{i=1}^{s}$, which requires $qsm$ entries. Hence, the total storage is $n+m$ entries, or $n+m+qsm$ entries if these submatrices are stored in advance. For the efficient implementation given in \cref{eff-RBSCD}, in addition to storing $x^{(k)}$, we also store (i) $E_{\tau^{(i)}}(A_{\tau^{(i)}}^{\T}A_{\tau^{(i)}})^{\dag}$ for $i=1,\ldots,s$, requiring $qsn$ entries, (ii) the matrices $A_{\tau^{(i)}}^{\T}A_{\tau^{(j)}}(A_{\tau^{(j)}}^{\T}A_{\tau^{(j)}})^{\dag}$ for $i,j=1,\ldots,s$, requiring $q^2s^2$ entries, and (iii) the auxiliary vectors $\{D_i^{k}\}_{i=1}^{s}$, requiring $qs$ entries. Therefore, the total storage is $n+qsn+q^2s^2+qs$ entries.
%Therefore, the additional storage required beyond $x^{(k)}$ is $qsn+q^2s^2+qs$ entries.
\section{Lemmas}\label{sec-lemmas}
To facilitate the convergence analysis of the RBSCD method, we present the following auxiliary lemmas, stated in the order they are used.
\begin{lemma}[\cite{BaiWu18SIAM}]
	\label{lem1}
	Let $A\in\reals^{m\times n}$ be a nonzero matrix. For any  $u\in\range(A^{\T})$, it holds that
	\begin{equation*}
		% \nonumber % Remove numbering (before each equation)
		\sigma_{\min}^2(A)\|u\|_2^2
		\leq \|Au\|_2^2
		\leq \sigma_{\max}^2(A)\|u\|_2^2.
	\end{equation*}
\end{lemma}

\begin{lemma}
	\label{lem3}
	Let $A\in\reals^{m\times n}$ be a nonzero matrix. Let $S\in\reals^{n\times n}$ be an invertible matrix and define $B=AS$, then the linear systems $A^{\T}y=A^{\T}b$ and $B^{\T}y=B^{\T}b$ have the same least-norm solution. Moreover,
	\begin{equation*}
		% \nonumber % Remove numbering (before each equation)
		(A^{\T})^{\dag}A^{\T}b
		=(B^{\T})^{\dag}B^{\T}b
		=\left[(AS)^{\T}\right]^{\dag}(AS)^{\T}b.
	\end{equation*}
\end{lemma}
\begin{proof}
Since $B^{\T}=S^{\T}A^{\T}$ and $S$ is nonsingular, the following two systems
\begin{equation*}
	A^{\T}y=A^{\T}b \quad \mbox{and} \quad B^{\T}y=B^{\T}b
\end{equation*}
are equivalent, and hence have the same solution set. Therefore, they share the same least-norm solution, which implies that
\begin{equation*}
	% \nonumber % Remove numbering (before each equation)
	(A^{\T})^{\dag}A^{\T}b
	=(B^{\T})^{\dag}B^{\T}b
	=\left[(AS)^{\T}\right]^{\dag}(AS)^{\T}b.
\end{equation*}
\end{proof}

\begin{lemma}
	\label{lem6}
	Let $A\in\reals^{m\times n}$ be a matrix with $\rank(A)=n$, and $S\in\mathbb{DR}^{n\times n}$. Let $\tau$ be a batch-sampling of size $q$ drawn from the index set $[n]$ satisfying the joint distribution $\bbb{P}$, where the entries $\tau_i$ are random variables with marginal distributions given in \cref{marDistrib}, $\{P_i\}_{i=1}^q$ be the diagonal matrices defined in \cref{marDistribMatrix}. Let $A_{\tau}$ be the submatrix of $A$ consisting of the columns indexed by $\tau$, $S_{\tau}$ be the principal submatrix including the diagonals of $S$ indexed by $\tau$. Then for any vector $r\in \reals^{m}$, it holds that
	\begin{align}
		\label{lem5Eq0}% \nonumber % Remove numbering (before each equation)
		\E\left\|(A_{\tau}^{\T})^{\dag}A_{\tau}^{\T}r\right\|_2^2&\geq \sum_{i=1}^{q}\left\|R_{S;i}^{-\frac{1}{2}} P_i^{\frac{1}{2}}S A^{\T}r\right\|_2^2,
	\end{align}
	where
	\begin{align}
		\label{lem5Eq3}% \nonumber % Remove numbering (before each equation)
		R_{S;i}=\diag\left(\gamma^S_{i,1},\gamma^S_{i,2},\ldots,\gamma^S_{i,n}\right),
		\quad\mbox{for $i=1,\ldots,q$,}
	\end{align}
	with
	\begin{align*}
		% \nonumber % Remove numbering (before each equation)
		\gamma^S_{i,j} &= \left\{\begin{array}{lr}
			\max_{\tau=(\tau_1,\ldots,\tau_{i-1},\tau_i=j,
				\tau_{i+1},\ldots,\tau_q)} \sigma_{\max}^2\left(A_{\tau} S_{\tau}\right), & \mbox{if $p_{ij}>0$}, \\
			\max_{\tau\sim\bbb{P}} \sigma_{\max}^2\left(A_{\tau} S_{\tau}\right), & \mbox{if $p_{ij}=0$},
		\end{array}
		\right.
		\quad\mbox{for $j=1,\ldots,n$}.
	\end{align*}
\end{lemma}
\begin{proof}
	Since $S$ is an invertible diagonal matrix, its submatrix $S_{\tau}$ is also a nonsingular diagonal matrix. Moreover, as $A_{\tau}^{\top} r \in \mathrm{range}(A_{\tau}^{\top})$, \cref{lem1,lem3} imply that
	\begin{align*}
		% \nonumber % Remove numbering (before each equation)
		\left\|(A_{\tau}^{\T})^{\dag}A_{\tau}^{\T}r\right\|_2^2
		&=\left\|[(A_{\tau}S_{\tau})^{\T}]^{\dag}(A_{\tau}S_{\tau})^{\T}r\right\|_2^2
		\geq \frac{1}{\sigma_{\max}^2\left(A_{\tau} S_{\tau}\right)}\left\|S_{\tau}A_{\tau}^{\T}r\right\|_2^2.
	\end{align*}
	Let $h_{\tau}=A_{\tau}^{\T}r$. Taking expectations and using the monotonicity of expectation, we obtain
	\begin{align}
		\nonumber % Remove numbering (before each equation)
		\E\left\|(A_{\tau}^{\T})^{\dag}A_{\tau}^{\T}r\right\|_2^2
		&\geq \E\left(\frac{1}{\sigma_{\max}^2\left(A_{\tau} S_{\tau}\right)}\left\|S_{\tau}A_{\tau}^{\T}r\right\|_2^2\right) \\ \nonumber
		&=\E\left(\sum_{i=1}^{q}|s_{\tau_i}|^2|h_{\tau_i}|^2\frac{1}{\sigma_{\max}^2\left(A_{\tau} S_{\tau}\right)}\right) \\ \label{lem5Eq1}
		&=\sum_{i=1}^{q}\E\left(|s_{\tau_i}|^2|h_{\tau_i}|^2\frac{1}{\sigma_{\max}^2\left(A_{\tau} S_{\tau}\right)}\right).
	\end{align}
	By the tower property of conditional expectation and the fact that
	$|s_{\tau_i}|^2|h_{\tau_i}|^2$ is $\sigma(\tau_i)$-measurable, for each $\E\left(|s_{\tau_i}|^2|h_{\tau_i}|^2\frac{1}{\sigma_{\max}^2\left(A_{\tau} S_{\tau}\right)}\right)$ in \cref{lem5Eq1}, it follows that
	\begin{align}
		\nonumber % Remove numbering (before each equation)
		\E\left(|s_{\tau_i}|^2|h_{\tau_i}|^2\frac{1}{\sigma_{\max}^2\left(A_{\tau} S_{\tau}\right)}\right)
		&= \E\left[\E\left( |s_{\tau_i}|^2|h_{\tau_i}|^2\left.\frac{1}{\sigma_{\max}^2\left(A_{\tau} S_{\tau}\right)}\right|_{\tau_i}\right)\right] \\ \nonumber
		&=\E\left[|s_{\tau_i}|^2|h_{\tau_i}|^2\E\left(\left.\frac{1}{\sigma_{\max}^2\left(A_{\tau}S_{\tau}\right)}\right|_{\tau_i}\right)\right] \\ \nonumber
		&=\sum_{j=1}^{n} |s_j|^2|h_j|^2 p_{ij}\E\left(\left.\frac{1}{\sigma_{\max}^2\left(A_{\tau} S_{\tau}\right)}\right|_{\tau_i=j}\right) \\ \nonumber
		%	\displaybreak
		&\geq \sum_{j=1}^{n} |s_j|^2|h_j|^2 p_{ij} \frac{1}{\gamma^S_{i,j}} \\ \nonumber
		&=\left\|R_{S;i}^{-\frac{1}{2}} P_i^{\frac{1}{2}}S h\right\|_2^2 \\ \label{lem5Eq2}
		&=\left\|R_{S;i}^{-\frac{1}{2}} P_i^{\frac{1}{2}}S A^{\T}r\right\|_2^2,\quad (\mbox{since $h=A^{\T}r$})
	\end{align}
	Therefore, \cref{lem5Eq0} follows directly from \cref{lem5Eq1} and \cref{lem5Eq2}.
\end{proof}

\begin{lemma}
	\label{lem7}
	Let $A\in\reals^{m\times n}$ be a matrix with $\rank(A)=n$. Let $\tau$ be a batch-sampling of size $q$ drawn from the index set $[n]$ satisfying the joint distribution $\bbb{P}$, where the entries $\tau_i$ are random variables with marginal distributions given in \cref{marDistrib}. Let $A_{\tau}$ be the submatrix of $A$ consisting of the columns indexed by $\tau$, and set $d_{\tau}=\dim\left[\range(A_{\tau})\right]$. Let $\mathcal{W}_{\tau}\subseteq\range(A)$ be a subspace such that $\mathcal{W}_{\tau}\perp\range(A_{\tau})$ and $\dim(\mathcal{W}_{\tau})=n-d_{\tau}$. Let $W_{\tau}\in\reals^{m\times (n-d_{\tau})}$ be a matrix whose columns form an orthonormal basis of $\mathcal{W}_{\tau}$. If $y\in\range(A)$ is a nonzero vector such that $y\perp \range(A_{\tau})$, then for any invertible matrix $M\in\reals^{n\times n}$, it holds that
	\begin{equation}
		\label{lem7Eq1}
		% \nonumber % Remove numbering (before each equation)
		\|MA^{\T}y\|_2^2
		\geq \xi_{\tau}\|y\|_2^2 \quad \mbox{with} \quad \xi_{\tau}=\lambda_{\min}\left( W_{\tau}^{\T}AM^{\T}MA^{\T}W_{\tau} \right).
	\end{equation}
\end{lemma}
\begin{proof}
	Let $U_{\tau}\in\reals^{m\times d_{\tau}}$ be a matrix whose columns form an orthonormal basis of $\range(A_{\tau})$. According to the definition of $W_{\tau}$, there exists an invertible matrix $F\in \reals^{n \times n}$ such that $A=[U_{\tau}\ W_{\tau}]\ F$ and $U^{\T}_{\tau}W_{\tau}=0$, which leads to a fact that
	\begin{align*}
		MA^{\T}W_{\tau}=MF^{\T}\begin{bmatrix}U_{\tau}^{\T}\\W_{\tau}^{\T}\end{bmatrix}W_{\tau}
		=MF^{\T}\begin{bmatrix}0\\I\end{bmatrix}=MF^{\T}_{-d_{\tau}},
	\end{align*}
	where $F^{\T}_{-d_{\tau}}$ denotes the submatrix obtained by deleting the first $d_{\tau}$ columns of $F^{\T}$. Since both $M$ and $F$ are invertible, it follows that $F^{\T}_{-d_{\tau}}$ has full column rank, so does $MF^{\T}_{-d_{\tau}}$, which implies that
	\begin{align*}
		W_{\tau}^{\T}AM^{\T}MA^{\T}W_{\tau}=(MA^{\T}W_{\tau})^{\T}MA^{\T}W_{\tau}=(MF^{\T}_{-d_{\tau}})^{\T}MF^{\T}_{-d_{\tau}}
	\end{align*}
	is positive definite. Since $y\in\range(A)$ is a nonzero vector such that $y\perp \range(A_{\tau})$, it holds that $y\in\mathcal{W}_{\tau}$. Accordingly, there exists a vector $z\in\reals^{n-d_{\tau}}$ such that $y=W_{\tau}z$. Let $\widehat{z}=z/\|z\|_2$, then  $\|\widehat{z}\|_2=1$, it follows that
	\begin{align}
		\label{lem7Eq2}
		\left\|MA^{\T}y\right\|_2^2
		&= \left\|MA^{\T}W_{\tau}z\right\|_2^2
		=\left\|MA^{\T}W_{\tau}\widehat{z}\right\|_2^2\|z\|^2_2 \nonumber\\
		&\geq \left(\min_{u\in\reals^{n-d_{\tau}},\ \|u\|_2=1}\left\|MA^{\T}W_{\tau}u\right\|_2^2\right)\|z\|^2_2\nonumber\\
		&=\lambda_{\min}\left(W_{\tau}^{\T}AM^{\T}MA^{\T}W_{\tau}\right)\|z\|^2_2.
	\end{align}
	Finally, since $W_{\tau}$ has orthonormal columns, we have $\|z\|_2=\|y\|_2$.
	Substituting this identity into \cref{lem7Eq2} yields \cref{lem7Eq1}.	
\end{proof}

\begin{lemma}[\cite{XieYang2025RBSK}]
	\label{lem8}
	Let $\tau$ be a batch-sampling of size $q$ drawn from the index set $[n]$ satisfying the joint distribution $\bbb{P}$. For each $j\in[n]$, denote by $\tau_{\txttiny{E$(j)$}}\in\range(\tau)$ the effective sets such that the column index $j\in\tau_{\txttiny{E$(j)$}}$, for $j=1,\ldots,n$, it follows that
	\begin{equation}
		\label{lem7EqResult}% \nonumber % Remove numbering (before each equation)
		\sum_{\tau_{\txttiny{E$(j)$}}\in\range(\tau)}\pr\left(\tau_{\txttiny{E$(j)$}}\right)
		\leq\sum_{i=1}^{q} p_{ij}
		\leq\sum_{\tau_{\txttiny{E$(j)$}}\in\range(\tau)}\left(q-|\tau_{\txttiny{E$(j)$}}|+1\right)\pr\left(\tau_{\txttiny{E$(j)$}}\right),
		%		\quad \mbox{for}\ j=1,\ldots,n,
	\end{equation}
	Moreover, the equalities hold in \cref{lem7EqResult} whenever $|\tau_{\txttiny{E$(j)$}}|=q$.
\end{lemma}

\begin{lemma}[\cite{XieYang2025RBSK}]
	\label{Hoeffding-co}
	Let $x_1, \dots, x_n$ be independent samples drawn from a distribution of a random variable $X$ supported on $[a, b]$ almost surely. Let $\bar{x}_n = \frac{1}{n} \sum_{i=1}^n x_i$ be the sample mean. For any $\varepsilon > 0$ and $\delta \in (0,1)$, if the number of samples satisfies
	\begin{equation*}
		% \nonumber % Remove numbering (before each equation)
		n \geq \frac{(b - a)^2}{2 \varepsilon^2} \log\left(\frac{2}{\delta}\right),
	\end{equation*}
	then it holds that
	\begin{equation*}
		\pr\left(\left|\bar{x}_n-\E(X)\right|\geq\epsilon\right)
		\leq \delta.
	\end{equation*}
\end{lemma}

\section{Convergence analysis of the RBSCD method}\label{RBSCD-analyse}
In this section, we present the convergence analysis of the RBSCD method. Our main result is \cref{convRBSCD}, which establishes a new expected linear convergence rate of RBSCD. \cref{rem:optimization,rem:learning} then explain the respective roles of the scaling matrix $S$ and the joint distribution $\bbb{P}$ in sharpening the bound and determining the iteration of RBSCD, while also clarify the scale-invariant property of the new convergence rate estimate. \cref{rem:compare-needell} further demonstrates that the new bound is sharper than the classical bounds presented in \cite{NeedellLAA15, WuNeedell2018RBGS}. We next extend the analysis to noisy least-squares problems in \cref{convRBSCD-noise}. Finally, \cref{convRBSCD_sharper} yields a tighter convergence rate estimate than \cref{convRBSCD} by employing concentration inequalities under slightly stronger assumptions. We now state the main result of this section.

\begin{theorem}
	\label{convRBSCD}
	Consider the linear least-squares problem \cref{least-squares-prob} with $A\in\reals^{m\times n}$ and $\rank(A)=n$. Let $\tau$ denote a batch-sampling of size $q$ drawn from the index set $[n]$  satisfying the joint distribution $\bbb{P}$, where the entries $\tau_i$ are random variables with marginal distributions given in \cref{marDistrib}. Let $A_{\tau}$ be the submatrix of $A$ consisting of the columns indexed by $\tau$, and set $d_{\tau}=\dim\left[\range(A_{\tau})\right]$. Let $\mathcal{W}_{\tau}\subseteq\range(A)$ be a subspace such that $\mathcal{W}_{\tau}\perp\range(A_{\tau})$ and $\dim(\mathcal{W}_{\tau})=n-d_{\tau}$. Let $W_{\tau}\in\reals^{m\times (n-d_{\tau})}$ be a matrix whose columns form an orthonormal basis of $\mathcal{W}_{\tau}$. Then for any initial guess $x^{(0)}\in\reals^{n}$, the iterative sequence $\{x^{(k)}\}_{k=0}^{+\infty}$ generated by the RBSCD method converges linearly in expectation to the unique least-squares solution $x_{LS}$ of \cref{least-squares-prob}. Moreover, the mean squared error measured in $A^{\T}A$-norm satisfies the following error estimate. For the initial step $k=0$,  we have
	\begin{equation}
		\label{thm1Eq1}%\nonumber % Remove numbering (before each equation)
		\E\|x^{(1)}-x_{LS}\|_{A^{\T}A}^2
		\leq \left[\min_{S\in\mathbb{DR}^{n\times n}}\left(1-\eta\right)\right]\|x^{(0)}-x_{LS}\|_{A^{\T}A}^2,
	\end{equation}
	and for any $k>0$,
	\begin{equation}
		\label{thm1Eq2}
		\E\|x^{(k+1)}-x_{LS}\|_{A^{\T}A}^2
		\leq \left[\min_{S\in\mathbb{DR}^{n\times 	n}}\Phi(\bbb{P},S)\right]\E\|x^{(k)}-x_{LS}\|_{A^{\T}A}^2.
	\end{equation}
	Consequently, it holds that
	\begin{equation}
		\label{thm1Eq3}% \nonumber % Remove numbering (before each equation)
		\E\|x^{(k+1)}-x_{LS}\|_{A^{\T}A}^2
		\leq \left[\min_{S\in\mathbb{DR}^{n\times n}}\Phi(\bbb{P},S)\right]^k
		\E\|x^{(1)}-x_{LS}\|_{A^{\T}A}^2,
	\end{equation}
	where the convergence factor is defined by
	\begin{equation*}
		\Phi(\bbb{P},S)=1-\xi,
	\end{equation*} 
	the quantities $\eta$ and $\xi$ are defined by
	\begin{equation*}
		% \nonumber % Remove numbering (before each equation)
		\eta =\lambda_{\min}\left(AM^2A^{\T}\right),
	\end{equation*}
	and
	\begin{equation*}
		% \nonumber % Remove numbering (before each equation)
		\xi = \min_{\tau\sim\bbb{P}} \lambda_{\min}\left(W_{\tau}^{\T}A M^2A^{\T} W_{\tau}\right),
	\end{equation*}
	with
	\begin{equation}
		\label{M_formula}
		M^2 = S^{\T}\left(\sum_{i=1}^{q} R^{-1}_{S;i} P_i\right)S,
	\end{equation}
	here, $S\in\mathbb{DR}^{n\times n}$ is arbitrary, $P_i\in\mathbb{R}^{n\times n}$ and $R_{S;i}\in\mathbb{R}^{n\times n}$ are the diagonal matrices defined in \cref{marDistribMatrix} and \cref{lem5Eq3}, respectively.
\end{theorem}
\begin{proof}
	Decompose $b$ as $b = b_{\mathcal{R}(A)} + b_{\mathcal{R}(A)^{\perp}}$, where $b_{\mathcal{R}(A)}=Ax_{LS}$, $b_{\mathcal{R}(A)^{\perp}}=b-A x_{LS}$. Since $A^{\top}(b-Ax_{LS})=0$, we have $A_{\tau^{(i_k)}}^{\top} b_{\mathcal{R}(A)^{\perp}} = 0$. According to the update rule of the RBSCD method, it follows that
	\begin{align}
		A(x^{(k+1)}-x_{LS})
		&=A[x^{(k)}+E_{\tau^{(i_k)}}A_{\tau^{(i_k)}}^{\dag}(b-Ax^{(k)})-x_{LS}]	\nonumber \\
		&=A[x^{(k)}+E_{\tau^{(i_k)}}A_{\tau^{(i_k)}}^{\dag}(b_{\mathcal{R}(A)}+b_{\mathcal{R}(A)^{\perp}}-Ax^{(k)})-x_{LS}] \nonumber \\
		&=Ax^{(k)}+A_{\tau^{(i_k)}}A_{\tau^{(i_k)}}^{\dag}(Ax_{LS}+b_{\mathcal{R}(A)^{\perp}}-Ax^{(k)})-Ax_{LS} \nonumber \\
		&=(I-A_{\tau^{(i_k)}}A_{\tau^{(i_k)}}^{\dag})A(x^{(k)}-x_{LS})+A_{\tau^{(i_k)}}A_{\tau^{(i_k)}}^{\dag}b_{\mathcal{R}(A)^{\perp}}\nonumber \\
		&=(I-A_{\tau^{(i_k)}}A_{\tau^{(i_k)}}^{\dag})A(x^{(k)}-x_{LS})\ ,
		\label{thm1Eq4}
	\end{align}
	where the last equality follows from the fact that $A_{\tau^{(i_k)}}A_{\tau^{(i_k)}}^{\dag}b_{\mathcal{R}(A)^{\perp}} =A_{\tau^{(i_k)}}(A^{\T}_{\tau^{(i_k)}}A_{\tau^{(i_k)}})^{\dag}A^{\T}_{\tau^{(i_k)}}b_{\mathcal{R}(A)^{\perp}} =0.$
	It can be readily verified that $A_{\tau^{(i_k)}}A_{\tau^{(i_k)}}^{\dag}$ is the orthogonal projector onto $\range(A_{\tau^{(i_k)}})$,  while $I-A_{\tau^{(i_k)}}A_{\tau^{(i_k)}}^{\dag}$ is the orthogonal projector onto $\range(A_{\tau^{(i_k)}})^{\perp}$. Therefore,
	\begin{equation*}
		(I-A_{\tau^{(i_k)}}A_{\tau^{(i_k)}}^{\dag})A(x^{(k)}-x_{LS}) \perp A_{\tau^{(i_k)}}A_{\tau^{(i_k)}}^{\dag}A(x^{(k)}-x_{LS}).
	\end{equation*}
	Applying the Pythagorean theorem to this orthogonal decomposition yields
	\begin{equation}
		\label{thm1Eq5}
		\|A(x^{(k)}-x_{LS})\|_2^2
		=\|(I-A_{\tau^{(i_k)}}A_{\tau^{(i_k)}}^{\dag})A(x^{(k)}-x_{LS})\|_2^2+\|A_{\tau^{(i_k)}}A_{\tau^{(i_k)}}^{\dag}A(x^{(k)}-x_{LS})\|_2^2.
	\end{equation}
	Combining \cref{thm1Eq4} and \cref{thm1Eq5}, and using $r^{(k)}=b-Ax^{(k)}=b_{\mathcal{R}(A)^{\perp}}+A(x_{LS}-x^{(k)})$, together with the basic properties of the pseudoinverse, it follows that
	\begin{align*}
		% \nonumber % Remove numbering (before each equation)
		\|A(x^{(k+1)}-x_{LS})\|_2^2
		&= \|A(x^{(k)}-x_{LS})\|_2^2 - \|A_{\tau^{(i_k)}}A_{\tau^{(i_k)}}^{\dag}r^{(k)}\|_2^2 \\
		&= \|A(x^{(k)}-x_{LS})\|_2^2 - \|(A_{\tau^{(i_k)}}A_{\tau^{(i_k)}}^{\dag})^{\T}r^{(k)}\|_2^2\\
		&= \|A(x^{(k)}-x_{LS})\|_2^2 - \|(A_{\tau^{(i_k)}}^{\T})^{\dag}A_{\tau^{(i_k)}}^{\T}r^{(k)}\|_2^2.
	\end{align*}
	
	Taking the conditional expectation with respect to the first $k$ iterations of the RBSCD method and applying \cref{lem6}, we obtain
	\begin{align}
		\E_k\|A(x^{(k+1)}-x_{LS})\|_2^2
		&= \|A(x^{(k)}-x_{LS})\|_2^2
		- \E_k\|(A_{\tau^{(i_k)}}^{\T})^{\dag}A_{\tau^{(i_k)}}^{\T}r^{(k)}\|_2^2 \notag\\
		&\le \|A(x^{(k)}-x_{LS})\|_2^2
		- \sum_{i=1}^{q}\|R_{S;i}^{-\frac{1}{2}} P_i^{\frac{1}{2}}SA^{\T}r^{(k)}\|_2^2 \notag\\
		&= \|A(x^{(k)}-x_{LS})\|_2^2
		- \sum_{i=1}^{q}\|R_{S;i}^{-\frac{1}{2}} P_i^{\frac{1}{2}}SA^{\T}A(x^{(k)}-x_{LS})\|_2^2 \notag\\
		&= \|A(x^{(k)}-x_{LS})\|_2^2
		- (x^{(k)}-x_{LS})^{\T}A^{\T}AS^{\T}
		\Big(\sum_{i=1}^{q}R_{S;i}^{-1} P_i\Big)
		SA^{\T}A(x^{(k)}-x_{LS}) \notag\\
		&= \|A(x^{(k)}-x_{LS})\|_2^2
		- \|MA^{\T}A(x^{(k)}-x_{LS})\|_2^2,
		\label{thm1Eq6}
	\end{align}
	where the diagonal matrix
	\begin{equation*}
		M = \left[S^{\T}\left(\sum_{i=1}^{q}R_{S;i}^{-1} P_i\right)S\right]^{\frac{1}{2}}
	\end{equation*}
	is positive definite, due to the fact that $S$ is invertible and diagonal, each $R_{S;i}$ is diagonally positive definite, and each $P_i$ is diagonally positive semi-definite satisfying the property \cref{marDistribMatrixProperty}.
	
	For the initial step $k=0$, since $M$ is an invertible matrix, the inclusion $A(x^{(0)}-x_{LS})\in\range(A)$ implies  $A(x^{(0)}-x_{LS})\in\range[(MA^{\T})^{\T}]$. According to \cref{lem1}, it follows from \cref{thm1Eq6} that
	\begin{equation*}
		%\nonumber % Remove numbering (before each equation)
		\E\|x^{(1)}-x_{LS}\|_{A^{\T}A}^2
		\leq \left(1-\eta\right)\|x^{(0)}-x_{LS}\|_{A^{\T}A}^2,
	\end{equation*}
	where
	\begin{equation*}
		\eta=\lambda_{\min}\left(AM^2A^{\T}\right).
	\end{equation*}
	Due to the fact that $S\in\mathbb{DR}^{n\times n}$ is arbitrary, the above estimate can be minimized with respect to $S$, which leads to the bound \cref{thm1Eq1}.
	
	For $k>0$, according to the update rule of the RBSCD method, the $k$th iterate $x^{(k)}$ is obtained by projecting the current iterate $x^{(k-1)}$ onto the solution space of the $(k-1)$th sampled normal equation $A_{\tau^{(i_{k-1})}}^{\T}Ax=A_{\tau^{(i_{k-1})}}^{\T}b$, it holds that
	\begin{equation*}
		% \nonumber % Remove numbering (before each equation)
		A_{\tau^{(i_{k-1})}}^{\T}(b-Ax^{(k)})=A_{\tau^{(i_{k-1})}}^{\T}(b_{\mathcal{R}(A)}-Ax^{(k)})= A_{\tau^{(i_{k-1})}}^{\T}A(x_{LS}-x^{(k)})=0,
	\end{equation*}
	which implies that $A(x^{(k)}-x_{LS}) \perp \range(A_{\tau^{(i_{k-1})}})$. Since $A(x^{(k)}-x_{LS})\in\range(A)$ is nonzero, and $A(x^{(k)}-x_{LS}) \perp \range(A_{\tau^{(i_{k-1})}})$, the relation \cref{thm1Eq6} together with \cref{lem7} leads to the following estimate for the convergence rate of the RBSCD method at the $k$th iteration
	\begin{equation}
		\label{thm1Eq7}%\nonumber % Remove numbering (before each equation)
		\E_k\|x^{(k+1)}-x_{LS}\|_{A^{\T}A}^2
		\leq\left(1-\xi_{\tau^{(i_{k-1})}}\right)\|x^{(k)}-x_{LS}\|_{A^{\T}A}^2,
	\end{equation}
	where
	\begin{equation}
		\label{thm1xidef}
		\xi_{\tau^{(i_{k-1})}}
		= \lambda_{\min}\left(W_{\tau^{(i_{k-1})}}^{\T}A M^2A^{\T}W_{\tau^{(i_{k-1})}}\right).
	\end{equation}
	Defining a constant $\xi>0$ independent of $k$ by
	\begin{equation*}
		\xi = \min_{\tau\sim\bbb{P}} \lambda_{\min}\left(W_{\tau}^{\T}A M^2A^{\T}W_{\tau}\right),
	\end{equation*}
	the estimate \cref{thm1Eq7} implies
	\begin{equation}
		\label{thm1Eq8}
		\E_k\|x^{(k+1)}-x_{LS}\|_{A^{\T}A}^2
		\leq \left(1-\xi\right)\|x^{(k)}-x_{LS}\|_{A^{\T}A}^2.
	\end{equation}
	Taking full expectation on both sides of \cref{thm1Eq8}, it holds that
	\begin{equation}
		\label{thm1Eq9}% \nonumber % Remove numbering (before each equation)
		\E\|x^{(k+1)}-x_{LS}\|_{A^{\T}A}^2
		\leq \left(1-\xi\right)\E\|x^{(k)}-x_{LS}\|_{A^{\T}A}^2,
	\end{equation}
	Since $S\in\mathbb{DR}^{n\times n}$ is arbitrary and $\Phi(\bbb{P},S)=1-\xi=1-\min_{\tau\sim\bbb{P}} \lambda_{\min}\left(W_{\tau}^{\T}A M^2A^{\T}W_{\tau}\right)$, minimizing the right-hand side of \cref{thm1Eq9} over $S$ results in \cref{thm1Eq2}. Iterating the recursion in \cref{thm1Eq2} from $k$ down to $1$ yields \cref{thm1Eq3}. 
\end{proof}

\begin{remark}
	\label{rem:optimization}
	For $k>0$, the convergence factor $\Phi(\bbb{P},S)$ defined in \cref{convRBSCD} provides a theoretical upper bound and can be viewed as a function of the sampling parameter $\bbb{P}$ and the scaling operator $S$. The scaling operator $S$ is an auxiliary quantity introduced solely for the convergence analysis and is neither constructed nor used in \cref{alg-RBSCD,eff-RBSCD}. Therefore, for a given sampling distribution $\bbb{P}$, modifying $S$ does not affect the actual convergence behavior of the algorithms, but only influences the tightness of the theoretical convergence rate bound. Therefore, computing the sharpest theoretical upper bound of the convergence rate for a given $\bbb{P}$ is equivalent to solving the optimization problem
	\begin{equation*}
		\Phi(\bbb{P},\widehat{S}_{\txt{opt}})
		=\min_{S\in\mathbb{DR}^{n\times n}} \Phi(\bbb{P},S),
	\end{equation*}
	where $\widehat{S}_{\txt{opt}}$ denotes an optimizer that yields the sharpest bound.
	
	Moreover, thanks to the introduction of the scaling operator $S$, the convergence bounds in \cref{convRBSCD} are invariant to the column scaling of the data matrix $A$.
\end{remark}

\begin{remark}
	\label{rem:learning}
	Based on the theoretical upper bound derived in \cref{convRBSCD}, constructing an RBSCD method with the fastest guaranteed convergence in theory can be formulated as the optimization problem
	\begin{equation}
		\label{upperBound_opt2}
		[\bbb{P}_{\txt{opt}},S_{\txt{opt}}]
		=\arg\min_{\bbb{P},\,S}\,\Phi(\bbb{P},S),
	\end{equation}
	where the pair $[\bbb{P}_{\txt{opt}},S_{\txt{opt}}]$ yields the sharpest upper bound on the convergence factor. Here $\bbb{P}_{\txt{opt}}$ represents an optimal joint distribution, corresponding to an optimal stochastic sampling rule. However, solving \cref{upperBound_opt2} is extremely challenging. Therefore, we simplify the problem by fixing the scaling operator $S=\widehat{S}$ (e.g., by taking $\widehat{S}=I$), and defining $\Psi(\bbb{P})=\Phi(\bbb{P},\widehat{S})$. In this way,  problem \cref{upperBound_opt2} reduces to
	\begin{equation}
		\label{upperBound_opt3}
		\widehat{\bbb{P}}_{\txt{opt}}
		=\arg\min_{\bbb{P}}\,\Psi(\bbb{P}).
	\end{equation}
	Even with a fixed scaling operator, solving \cref{upperBound_opt3} remains difficult for large scale least-squares problems. As a practical approach to addressing this difficulty, a variety of greedy and adaptive sampling rules have been proposed to enhance convergence by updating the sampling distribution using the information from the current residual \cite{BaiWu19NLAA,LiZhang2020GBGS,LiuJiangGu2021MRBGS,gower2021adaptive}. Nevertheless, such rules often increase the per-iteration cost because they require additional calculations and extra access to the data. Motivated by these observations, we consider a learning-based sampling strategy within the RBSCD iterations. Specifically, for a given class of application problems, we parameterize the sampling distribution $\bbb{P}$ and optimize its parameters through gradient-based training on a dataset of such instances to obtain a distribution that yields faster convergence and higher computational efficiency, and then use the learned distribution in the RBSCD iterations when solving new problems from the same class. This strategy maintains the low per-iteration complexity of classical randomized methods, while exploiting problem-specific structure to achieve faster convergence.
	
	Preliminary experiments on CT reconstruction problems confirm the efficacy of this strategy. In our experiments, a learned distribution leads to faster convergence and improved runtime performance compared with standard baselines such as uniform sampling and sampling proportional to the squared block $\ell_2$-norms. Moreover, when compared with several efficient block coordinate descent solvers that employ greedy index sampling rules, the learned-distribution RBSCD method demonstrates competitive performance. In particular, it achieves lower computational cost to reach a prescribed accuracy level, indicating improved practical efficiency. Detailed numerical results are reported in \cref{section 5.2,section 5.3}.
\end{remark}	

\begin{remark}
	\label{rem:compare-needell}
	In order to facilitate a direct comparison with the classical convergence bounds established by Needell et al. {\rm{\cite{NeedellLAA15,WuNeedell2018RBGS}}} for the RBCD method under column-paving batch-sampling, we derive a relaxed upper bound corresponding to \cref{convRBSCD}. Let $\gamma^S=\max_{\tau\sim\bbb{P}}\|A_{\tau}S_{\tau}\|_2^2$, by the definition of $R_{S;i}$ in \cref{lem5Eq3}, it follows for each $i=1,\ldots,q$ that $R_{S;i}^{-1} \succeq \frac{1}{\gamma^S}I.$ Recall the matrix $M^2$ defined in \cref{convRBSCD}, together with \cref{lem8}, we obtain that
	\begin{equation*}
		%		\label{eq:relax-M2}
		M^2
		\;=\; S^{\T}\Big(\sum_{i=1}^{q} R_{S;i}^{-1}P_i\Big)S \;\succeq\; \frac{1}{\gamma^S}S^{\T}\Big(\sum_{i=1}^{q} P_i\Big)S
		\;=\; \frac{1}{\gamma^S}S^{\T}\widetilde{P} S \overset{\text{\cref{lem8}}}{\succeq} \frac{1}{\gamma^S}S^{\T}\widehat P S,
	\end{equation*}
	where the diagonal matrices $\widetilde{P}$ and $\widehat{P}$ are defined by their diagonal entries
	\begin{equation*}
		\widetilde{p}_j=\sum_{i=1}^{q} p_{ij},\quad
		\widehat{p}_j=\sum_{\tau_{\txttiny{E$(j)$}}\in\range(\tau)}\pr\left(\tau_{\txttiny{E$(j)$}}\right)\quad
		\mbox{for $j=1,\ldots,n,$}
	\end{equation*}
	such that $\widetilde{P}=\diag(\widetilde{p}_1,\widetilde{p}_2,\ldots,\widetilde{p}_n)$, $\widehat{P}=\diag(\widehat{p}_1,\widehat{p}_2,\ldots,\widehat{p}_n)$.
	%	\begin{align*}
		%		% \nonumber % Remove numbering (before each equation)
		%		\widetilde{P} &= \begin{bmatrix}
			%			\sum_{i=1}^{q} p_{i1} &   &   \\
			%			& \ddots &   \\
			%			&   & \sum_{i=1}^{q} p_{in}
			%		\end{bmatrix},
		%	\end{align*}
	%	and
	%	\begin{align*}
		%		\nonumber % Remove numbering (before each equation)
		%		\widehat{P} &= \begin{bmatrix}
			%			\sum_{\tau_{\txttiny{E$(1)$}}\in\range(\tau)} \pr\left(\tau_{\txttiny{E$(1)$}}\right)   &   &   \\
			%			& \ddots &   \\
			%			&   & \sum_{\tau_{\txttiny{E$(n)$}}\in\range(\tau)} \pr\left(\tau_{\txttiny{E$(n)$}}\right)
			%		\end{bmatrix}.
		%	\end{align*}
	Using the monotonicity of $\lambda_{\min}(\cdot)$ under the Loewner order, we obtain
	the relaxed upper bounds
	\begin{equation}
		\label{eq:relax-eta}
		1-\eta
		%		&=1-\lambda_{\min}(AM^2A^{\T})
		\leq1-\frac{1}{\gamma^S}\lambda_{\min}(AS^{\T}\widehat P SA^{\T}),
	\end{equation}
	and
	\begin{equation}
		\label{eq:relax-xi}
		1-\xi
		%		&=1-\min_{\tau\sim\bbb{P}}\lambda_{\min}\!\big(W_{\tau}^{\T}AM^2A^{\T}W_{\tau}\big)\nonumber
		\leq 1-\frac{1}{\gamma^S}\min_{\tau\sim\bbb{P}}
		\lambda_{\min}\!\big(W_{\tau}^{\T}AS^{\T}\widehat P SA^{\T}W_{\tau}\big),
	\end{equation}
	where $\eta$ and $\xi$ are defined in \cref{convRBSCD}.
	
	In particular, we now consider a batch-sampling $\tau$ with a prescribed joint distribution $\bbb{P}$, which corresponds to a column paving of $A$. The RBCD method proposed in {\rm{\cite{NeedellLAA15}}} is constructed by selecting a batch $\tau$ uniformly at random from a partition $\mathcal{T}=\left\{T_1, \ldots, T_{s}\right\}$ of the column indices $[n]$. In this setting, the effective sets $\tau_{\txttiny{E}}$ are sampled uniformly with equal probability $\frac{1}{s}$, and hence the matrix $\widehat P$ reduces to $\widehat P = \frac{1}{s}I$. Choosing the scaling matrix as $S=I$ yields $\gamma^I=\max_{\tau_{\txttiny{E}}\in\mathcal{T}}\|A_{\tau_{\txttiny{E}}}\|_2^2$. Consequently, for the initial step $k=0$, the bound \cref{eq:relax-eta} leads to a relaxed upper bound
	\begin{equation}
		\label{rlx_upperBound_1}
		% \nonumber
		1-\frac{1}{\gamma^I}\lambda_{\min}(AA^{\T}),
	\end{equation}
	while for $k>0$, the bound \cref{eq:relax-xi} leads to a relaxed upper bound
	\begin{equation}
		\label{rlx_upperBound_2}% \nonumber % Remove numbering (before each equation)
		1-\frac{1}{\gamma^I} \min_{\tau_{\txttiny{E}}\in\mathcal{T}}\lambda_{\min}\left(W_{\tau_{\txttiny{E}}}^{\T}AA^{\T}W_{\tau_{\txttiny{E}}}\right).
	\end{equation}
	The relaxed upper bound \cref{rlx_upperBound_1} recovers the convergence upper bound established by Needell et al. for the RBCD method. Moreover, owing to the presence of $W_{\tau}$, the  bound \cref{rlx_upperBound_2} for $k>0$ is sharper than \cref{rlx_upperBound_1}. As a result, with the arbitrariness of the scaling matrix $S$ and combining \cref{eq:relax-eta} and \eqref{eq:relax-xi}, the convergence upper bounds \cref{thm1Eq1} and \cref{thm1Eq2} presented in \cref{convRBSCD} are typically sharper than  \cref{rlx_upperBound_1} as well as the bound proved in {\rm{\cite{WuNeedell2018RBGS}}}.	
\end{remark}

Following the noisy linear system model introduced by Needell \cite{Needell2010BIT}, we consider a consistent linear system
$Ax_{\star}=b$ perturbed by additive noise. Specifically, let
$\bar b := b + w = A x_{\star} + w,$ where $w\in\mathbb{R}^{m}$ denotes a noise vector. The corresponding noisy least-squares problem is given by $\min_{x\in\mathbb{R}^{n}}\|Ax-\bar b\|_2^{2}$. Although \cref{convRBSCD} guarantees the convergence of RBSCD to the least-squares solution $x_{\mathrm{LS}}$, this solution generally differs from the underlying noise-free solution $x_\star$, which is often the quantity of primary interest in practice.  It is therefore important to characterize the behavior of the iterates relative to $x_\star$ and to quantify the effect of the noise. We next present \cref{convRBSCD-noise}, which characterizes this behavior. Specifically, it shows that the initial-error term in the mean squared error decays with the same linear convergence factor as in \cref{convRBSCD}, while the noise determines the resulting convergence horizon.
\begin{theorem}
	\label{convRBSCD-noise}
	Let $A\in\mathbb{R}^{m\times n}$ be a matrix with $\mathrm{rank}(A)=n$. Assume that $b\in\mathbb{R}^m$ is the right-hand side of a consistent linear system, so that there exists $x_{\star}\in\mathbb{R}^n$ satisfying $Ax_{\star}=b$.
	Let $\bar{b} = b + w$, where $w\in\mathbb{R}^m$ is a noise vector. Consider the following noisy least-squares problem
	\begin{equation*}
		\min_{x\in\mathbb{R}^n}\|Ax-\bar{b}\|_2^2,
	\end{equation*}
	whose unique minimizer is denoted by $x_{LS}$. The remaining conditions and notation are the same as in \cref{convRBSCD}. Then, for any initial guess $x^{(0)}\in \reals^{n}$, the iterative sequence $\{x^{(k)}\}_{k=0}^{+\infty}$ generated by the RBSCD method converges linearly in expectation to a neighborhood of $x_{\star}$, which is referred to as the convergence horizon. A geometric interpretation of a single RBSCD update for the noisy problem is provided in
	\cref{fig:my_label0}. Moreover, the mean squared error measured in $A^{\T}A$-norm satisfies
	\begin{equation}
		\label{thm2Eq1}
		\E\|x^{(k+1)}-x_{\star}\|_{A^{\T}A}^2
		\leq \bigl[\min_{S\in\mathbb{DR}^{n\times n}}\Phi(\bbb{P},S)\bigl]^k
		\bigl[\min_{S\in\mathbb{DR}^{n\times n}}\left(1-\eta\right)\bigl]\|x^{(0)}-x_{\star}\|_{A^{\T}A}^2+\mathcal{H}_{k}(\|w\|_2)
	\end{equation}
	where the convergence horizon is given by
	\begin{equation*}
		\mathcal{H}_{k}(\|w\|_2)=\left[\min_{S\in\mathbb{DR}^{n\times n}}\left(\frac{4-3\xi}{\xi}\right)+\zeta_1(k)\right]\,\|w\|_2^2 +\zeta_2(k)\,\|w\|_2,
	\end{equation*}
	and $\zeta_1(k)\ge 0$, $\zeta_2(k)\ge 0$ are deterministic sequences satisfying $\zeta_1(k)\to 0$ and $\zeta_2(k)\to 0$ as $k\to\infty$. Here, $\Phi(\bbb{P},S)$, $\eta$ and $\xi$ are defined as in \cref{convRBSCD}.
\end{theorem}

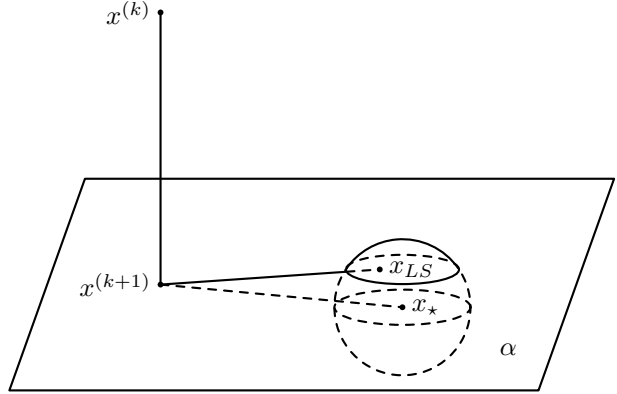
\begin{figure}[H]
	\centering
	\begin{tikzpicture}[>=Stealth, line cap=round, line join=round]
		% 1) Plane alpha
		\coordinate (P1) at (0,0);
		\coordinate (P2) at (7,0);
		\coordinate (P3) at (8,2.8);
		\coordinate (P4) at (1,2.8);
		\draw[thick] (P1)--(P2)--(P3)--(P4)--cycle;
		\node at (6.6,0.55) {$\alpha$};
		
		% 2) Points: x^{(k+1)} on plane, x^{(k)} above it
		\coordinate (xkp1) at (2.0,1.40);   % x^{(k+1)}
		\coordinate (xk)   at (2.0,5.00);   % x^{(k)}
		\fill (xkp1) circle (1.1pt);
		\fill (xk)   circle (1.1pt);
		\node[left]  at (xkp1) {$x^{(k+1)}$};
		\node[left]  at (xk)   {$x^{(k)}$};
		% dotted line indicating projection
		\draw[thick] (xkp1) -- (xk);
		
		% 3) x_{LS}, x^{(k+1)} on the plane, xstar not lying on the plane
		\coordinate (xLS) at (4.9,1.60);
		\fill (xLS) circle (1.1pt);
		\node[right] at (xLS) {$x_{LS}$};
		
		\coordinate (xstar) at (5.2,1.10);
		\fill (xstar) circle (1.1pt);
		\node[right] at (xstar) {$x_{\star}$};
		
		% 4) ball: circle and  ellipse
		% circle: upper solid, lower dashed
		\def\r{0.9}
		\def\theta{75} % 透视倾斜角
		\pgfmathsetmacro{\ry}{\r*cos(\theta)}
		\def\h{0.5}   % 圆心到平面距离, 0<h<R
		\pgfmathsetmacro{\rho}{sqrt(\r*\r - \h*\h)}
		\pgfmathsetmacro{\rhoy}{\rho*cos(\theta)}
		% 上面 1/3：30° 到 150° 实线
		\draw[thick]
		($(xstar)+(35:\r)$)
		arc[start angle=35, end angle=145, radius=\r];
		% 下面 2/3：150° 到 390° 虚线
		\draw[thick, dashed, dash phase=2pt]
		($(xstar)+(145:\r)$)
		arc[start angle=145, end angle=395, radius=\r];
		
		% ellipse xLS: upper dashed, lower solid
		\coordinate (xell) at ($(xstar)+(0,\h)$); %椭圆中心
		\draw[thick, dashed]
		($(xell)+(\rho,0)$)
		arc[start angle=0, end angle=180, x radius=\rho, y radius=\rhoy];
		\draw[thick]
		($(xell)+(-\rho,0)$)
		arc[start angle=180, end angle=360, x radius=\rho, y radius=\rhoy];
		
		% ellipse ball cross
		\draw[thick, dashed]
		($(xstar)+(\r,0)$)
		arc[start angle=0, end angle=360, x radius=\r, y radius=\ry];
		
		% 5) line (x^{k+1}--xLS, x^{k+1}--xstar)
		% 计算椭圆与直线交点
		\path[name path=linepath] (xkp1) -- (xLS);
		\path[name path=ellipsepath]
		($(xell)+(-\rho,0)$)
		arc[start angle=180, end angle=360, x radius=\rho, y radius=\rhoy];
		\path[name intersections={of=linepath and ellipsepath, by=I}];
		% 画分割直线
		\draw[thick] (xkp1) -- (I);
		\draw[thick, dashed] (I) -- (xLS);
		%		
		%		% x_\star on the boundary of the ball
		%		\coordinate (xstar) at ($(xLS)+(240:\r)$);
		%		\fill (xstar) circle (1.1pt);
		\draw[dashed, thick] (xkp1) -- (xstar);	
	\end{tikzpicture}
	\caption{Geometric interpretation of a single RBSCD update for a noisy least-squares problem. The current iterate $x^{(k)}$ is projected to the affine subspace $\alpha:=\{x:\ A_{\tau^{(i_k)}}^{\T}A x = A_{\tau^{(i_k)}}^{\T}\bar b\}$ along $\range(E_{\tau^{(i_k)}})$, yielding the next iterate $x^{(k+1)}$. The point $x_{LS}\in\alpha$ represents the noisy least-squares solution. The ball centered at $x_{\star}$ represents the convergence horizon, in which $x_{LS}$ lies.}
	\label{fig:my_label0}
\end{figure}
\begin{proof}
	Define $\Delta x=x_{LS}-x_{\star}$. By the triangle inequality, we obtain that
	\begin{align}
		\label{thm2Eq2}
		\|x^{(k+1)}-x_\star\|_{A^{\T}A}^2
		&=\|(x^{(k+1)}-x_{LS})+\Delta x\|_{A^{\T}A}^2 \nonumber\\
		&\leq \|x^{(k+1)}-x_{LS}\|_{A^{\T}A}^2
		+2\|x^{(k+1)}-x_{LS}\|_{A^{\T}A}\,\|\Delta x\|_{A^{\T}A}
		+\|\Delta x\|_{A^{\T}A}^2.
	\end{align}
	For $k>0$, taking expectation on both sides of \cref{thm2Eq2} and using \cref{thm1Eq9} yields
	\begin{multline}
		\label{thm2Eq3}
		\E\|x^{(k+1)}-x_\star\|_{A^{\T}A}^2
		%	&\leq\E\left(\|x^{(k+1)}-x_{LS}\|_{A^{\T}A}^2\right)+2\E\left(\|x^{(k+1)}-x_{LS}\|_{A^{\T}A}\right)\|\Delta x\|_{A^{\T}A}+ \|\Delta x\|_{A^{\T}A}^2 \\
		\leq(1-\xi)\E\|x^{(k)}-x_{LS}\|_{A^{\T}A}^2\\
		+2\E\bigl(\|x^{(k+1)}-x_{LS}\|_{A^{\T}A}\bigr)\|\Delta x\|_{A^{\T}A}+\|\Delta x\|_{A^{\T}A}^2.
	\end{multline}
	In addition, by the triangle inequality and the monotonicity of expectation, it follows that
	\begin{align}
		\E\|x^{(k)}-x_{LS}\|_{A^{\T}A}^2 &\leq\E\|x^{(k)}-x_\star\|_{A^{\T}A}^2+\|\Delta x\|_{A^{\T}A}^2+2\E\bigl(\|x^{(k)}-x_\star\|_{A^{\T}A}\bigr)\|\Delta x\|_{A^{\T}A}\nonumber \\
		&=\E\|x^{(k)}-x_\star\|_{A^{\T}A}^2+3\|\Delta x\|_{A^{\T}A}^2+2\E\bigl(\|x^{(k)}-x_{LS}\|_{A^{\T}A}\bigr)\|\Delta x\|_{A^{\T}A}. \label{thm2Eq4}
	\end{align}	
	Substituting \cref{thm2Eq4} into \cref{thm2Eq3}, then for $k>0$, it holds that
	\begin{multline}
		\label{thm2Eq5}
		\E\|x^{(k+1)}-x_\star\|_{A^{\T}A}^2\leq(1-\xi)\E\|x^{(k)}-x_{\star}\|_{A^{\T}A}^2
		+(4-3\xi)\|\Delta x\|_{A^{\T}A}^2 \\
		+2\left[(1-\xi)\E\|x^{(k)}-x_{LS}\|_{A^{\T}A}+\E\|x^{(k+1)}-x_{LS}\|_{A^{\T}A}\right]\|\Delta x\|_{A^{\T}A}.
	\end{multline}
	Let $\widetilde{\eta}=\max_{S}\eta$, $\widetilde{\xi}=\max_{S}\xi$, where $S$ is an arbitrary and invertible diagonal matrix, then $\min_{S}(1-\eta)=1-\widetilde{\eta}$ and $\min_{S}\Phi(\bbb{P},S)=\min_{S}(1-\xi)=1-\widetilde{\xi}$.
	Therefore, \cref{thm1Eq3} implies that, for all $k\ge0$,
	\begin{align}
		\E\|x^{(k+1)}-x_{LS}\|_{A^{\T}A}^2
		&\leq (1-\widetilde{\xi})^k(1-\widetilde{\eta})\|x^{(0)}-x_{LS}\|_{A^{\T}A}^2. \label{thm2Eq6}
	\end{align}
	By Jensen's inequality, for any random variable $z$, $\E(z)\leq\sqrt{\E(z^2)}$. Combining this fact with \cref{thm2Eq6}, there exists a constant $C_1>0$, independent of $k$, such that
	\begin{align*}
		2\left[(1-\widetilde{\xi})\E\|x^{(k)}-x_{LS}\|_{A^{\T}A}+\E\|x^{(k+1)}-x_{LS}\|_{A^{\T}A}\right]
		&\leq C_1(1-\widetilde{\xi})^{k/2},
		\quad k>0.
	\end{align*}
	Since \cref{thm2Eq5} holds for every admissible $S$, we may choose $S$ so that $\xi=\widetilde{\xi}$, therefore, for $k>0$,
	\begin{align}
		\E\|x^{(k+1)}-x_{\star}\|_{A^{\T}A}^2
		\leq
		(1-\widetilde{\xi})\E\|x^{(k)}-x_{\star}\|_{A^{\T}A}^2
		+(4-3\widetilde{\xi})\|\Delta x\|_{A^{\T}A}^2
		+C_1(1-\widetilde{\xi})^{k/2}\,\|\Delta x\|_{A^{\T}A}.
	\end{align}
	By a similar argument for initial step $k=0$, there exists a constant $C_2>0$, such that
	\begin{align*}
		\E\|x^{(1)}-x_{\star}\|_{A^{\T}A}^2
		&\leq
		(1-\widetilde{\eta})\,\|x^{(0)}-x_{\star}\|_{A^{\T}A}^2
		+(4-3\widetilde{\eta})\,\|\Delta x\|_{A^{\T}A}^2
		+C_2\,\|\Delta x\|_{A^{\T}A}.
	\end{align*}
	Iterating the above recursion from $k$ down to $0$ yields
	\begin{align}
		\label{thm2Eq7}
		\E\|x^{(k+1)}-x_{\star}\|_{A^{\T}A}^2\leq
		(1-\widetilde{\xi})^k(1-\widetilde{\eta})\|x^{(0)}-x_{\star}\|_{A^{\T}A}^2
		+\left(\frac{4-3\widetilde{\xi}}{\widetilde{\xi}}+\zeta_1(k)\right)\|\Delta x\|_{A^{\T}A}^2
		+\zeta_2(k)\|\Delta x\|_{A^{\T}A},
	\end{align}
	where
	\begin{equation*}
		\zeta_1(k)
		=(1-\widetilde{\xi})^k(4-3\widetilde{\eta}),
		\quad\text{and}\quad
		\zeta_2(k)
		=\frac{C_1}{1-\sqrt{1-\widetilde{\xi}}}\,(1-\widetilde{\xi})^{k/2}+C_2(1-\widetilde{\xi})^k.
	\end{equation*}
	Finally, by the following fact of least-squares,
	\begin{equation*}
		Ax_{LS}=\bar b_{\mathcal{R}(A)}=(b+w)_{\mathcal{R}(A)}=b+w_{\mathcal{R}(A)}.
	\end{equation*}
	It follows that
	\begin{equation}
		\label{Delta-noise-proj}
		\|\Delta x\|_{A^{\T}A}
		=\|A(x_{LS}-x_{\star})\|_2
		=\|w_{\mathcal{R}(A)}\|_2
		\leq \|w\|_2,
	\end{equation}
	which implies that the convergence horizon is controlled by the noise level. Combining \cref{thm2Eq7} and \cref{Delta-noise-proj}, the estimate \cref{thm2Eq1} is obtained.
\end{proof}

Building on \cref{convRBSCD}, whose convergence estimate relies on the uniform lower bound $\min_{\tau\sim\bbb{P}}\xi_{\tau}$ across all sampled batches and thus corresponds to a worst-case analysis, we aim to refine the convergence factor in the error estimate and derive a tighter upper bound for the RBSCD method via concentration inequalities under slightly stronger assumptions. This refinement is based on the intuition of the law of large numbers. Specifically, for a sufficiently large sample size, the sample mean provides a close approximation to the expectation of the random variable. However, the sample mean is subject to posterior nature and fluctuation, whose observed value can only be obtained after the results of the random experiment are known. Therefore, we express the resulting theoretical rate bound in terms of the expectation, whose value is a priori known and constant.
\begin{theorem}
	\label{convRBSCD_sharper}
	Under the same assumptions as in \cref{convRBSCD}, for $k>0$, consider a sample sequence $\left\{\xi_{\tau^{(i_{k-1},j)}}\right\}_{j=1}^{\ell}$ of the random variable
	$\xi_{\tau^{(i_{k-1})}}$
	%	$\xi_{\tau}$
	in \cref{thm1xidef}, generated from independent realizations $\left\{\tau^{(i_{k-1},j)}\right\}_{j=1}^{\ell}$. Correspondingly, let $\bigl\{\|x^{(k,j)}-x_{LS}\|^2_{A^{\top}A}\bigr\}_{j=1}^{\ell}$ be the sample sequence of $\|x^{(k)}-x_{LS}\|^2_{A^{\top}A} $, where $ x^{(k, j)} $ is the $k$th iterate obtained from the  $j$th realization. For any $\varepsilon>0, \delta \in (0,1)$, if the sample covariance satisfies $\cov\,(\xi_{\tau^{(i_{k-1})}},\|x^{(k)}-x_{LS}\|^2_{A^{\T}A})\geq 0$, and the sample size $\ell$ obeys
	\begin{equation*}
		% \nonumber % Remove numbering (before each equation)
		\ell \, \geq \, \frac{\Delta^2}{2 \varepsilon^2} \log\left(\frac{2}{\delta}\right),
	\end{equation*}
	where $\Delta=\lambda_{\max}(AM^2A^{\T}) - \lambda_{\min}(AM^2A^{\T})$.
	Then the mean squared error satisfies
	\begin{equation}
		\label{thm4Eq2}
		\E\|x^{(k+1)}-x_{LS}\|_{A^{\T}A}^2
		\leq \left[\min_{S\in\mathbb{DR}^{n\times n}}\left(1+\varepsilon-\E\left(\xi_{\tau}\right)\right)   \right]\E\|x^{(k)}-x_{LS}\|_{A^{\T}A}^2,
	\end{equation}	
	with probability at least $1-\delta$, moreover, if $\epsilon,\delta\to 0$, the mean squared error holds that
	\begin{equation}\label{thm4Eq2_1}
		\E\|x^{(k+1)}-x_{LS}\|_{A^{\T}A}^2
		\leq \left[\min_{S\in\mathbb{DR}^{n\times n}}\left(1-\E\left(\xi_{\tau}\right)\right)   \right]\E\|x^{(k)}-x_{LS}\|_{A^{\T}A}^2,
	\end{equation}
	where $\xi_{\tau} =\lambda_{\min}(W_{\tau}^{\T}A M^2A^{\T}W_{\tau})$ with $M^2$ given in \cref{M_formula}, and $S\in\mathbb{DR}^{n\times n}$ is arbitrary.
\end{theorem}
\begin{proof}
	Based on the inequality \cref{thm1Eq7} in the proof of \cref{convRBSCD} and the definition of the sample sequence, for $j=1,\ldots,\ell$, we obtain that
	\begin{equation}
		\label{thm4Eq3}%\nonumber % Remove numbering (before each equation)
		\E_k\|x^{(k+1,j)}-x_{LS}\|_{A^{\T}A}^2
		\leq\left(1-\xi_{\tau^{(i_{k-1},j)}}\right)\|x^{(k,j)}-x_{LS}\|_{A^{\T}A}^2
	\end{equation}
	with
	\begin{equation*}
		% \nonumber % Remove numbering (before each equation)
		\xi_{\tau^{(i_{k-1},j)}}
		= \lambda_{\min}\left(W_{\tau^{(i_{k-1},j)}}^{\T}A M^2A^{\T}W_{\tau^{(i_{k-1},j)}}\right).
	\end{equation*}
	By taking the sample mean on both sides of inequality \cref{thm4Eq3} over $j=1,\ldots,\ell$ yields
	\begin{equation}
		\label{thm4Eq4}%\nonumber % Remove numbering (before each equation)
		\frac{1}{\ell}\sum_{j=1}^{\ell}\E_k\|x^{(k+1,j)}-x_{LS}\|_{A^{\T}A}^2
		\leq\frac{1}{\ell}\sum_{j=1}^{\ell}\left[\left(1-\xi_{\tau^{(i_{k-1},j)}}\right)\|x^{(k,j)}-x_{LS}\|_{A^{\T}A}^2\right].
	\end{equation}
	For brevity, set $\xi_j=\xi_{\tau^{(i_{k-1},j)}}$, $f_j=\|x^{(k, j)}-x_{LS}\|^2_{A^{\T}A}$. Let $\bar{\xi}$ and $\bar{f}$ be the sample means of the sample sequences $\left\{\xi_{\tau^{(i_{k-1},j)}}\right\}_{j=1}^{\ell}$ and $\bigl\{\|x^{(k, j)}-x_{LS}\|^2_{A^{\T}A}\bigr\}_{j=1}^{\ell}$, which are defined by
	$\bar{\xi}=\frac{1}{\ell}\sum_{j=1}^{\ell}\xi_j$, $\bar{f}=\frac{1}{\ell}\sum_{j=1}^{\ell}f_j.$
	Since the sample covariance satisfies $\cov\,(\xi_{\tau^{(i_{k-1})}},\|x^{(k)}-x_{LS}\|^2_{A^{\T}A})\geq0$, it implies that
	\begin{align*}
		\frac{1}{\ell-1}\sum_{j=1}^{\ell}(\xi_j-\bar{\xi})(f_j-\bar{f})
		&=\frac{1}{\ell-1}\left(	\sum_{j=1}^{\ell}\xi_j f_j
		-\bar{\xi}\sum_{j=1}^{\ell}f_j
		-\bar{f}\sum_{j=1}^{\ell}\xi_j
		+\ell\bar{\xi}\bar{f}\right) \\
		&=\frac{1}{\ell-1}\left(	\sum_{j=1}^{\ell}\xi_j f_j
		-\ell\bar{\xi}\bar{f}\right) \geq 0,
	\end{align*}
	which leads to a fact that
	\begin{equation}
		\label{thm4Eq5}
		\frac{1}{\ell}\sum_{j=1}^{\ell}\xi_j f_j \ \geq \  \bar{\xi}\bar{f}.
	\end{equation}
	Substituting inequalities \cref{thm4Eq5} into \cref{thm4Eq4} gives
	\begin{equation}
		\label{thm4Eq6}%\nonumber % Remove numbering (before each equation)
		\frac{1}{\ell}\sum_{j=1}^{\ell}\E_k\|x^{(k+1,j)}-x_{LS}\|_{A^{\T}A}^2
		\leq\left(1-\frac{1}{\ell}\sum_{j=1}^{\ell}\xi_{\tau^{(i_{k-1},j)}}\right)\left(\frac{1}{\ell}\sum_{j=1}^{\ell}\|x^{(k,j)}-x_{LS}\|_{A^{\T}A}^2\right).
	\end{equation}
	Moreover, by the definition of $\xi_{\tau^{(i_{k-1},j)}}$, it reads that
	\begin{equation*}
		%\nonumber % Remove numbering (before each equation)
		\lambda_{\min}\left(AM^2A^{\T}\right)
		\leq \xi_{\tau^{(i_{k-1},j)}}
		\leq \lambda_{\max}\left(AM^2A^{\T}\right).
	\end{equation*}
	Applying \cref{Hoeffding-co} with the support interval
	$\bigl[\lambda_{\min}(AM^2A^{\T}),\,\lambda_{\max}(AM^2A^{\T})\bigr]$, and using the condition on sample size $\ell$, it follows that
	\begin{equation*}
		\pr\left(\left|\frac{1}{\ell}\sum_{j=1}^{\ell}\xi_{\tau^{(i_{k-1},j)}}-\E(\xi_{\tau^{(i_{k-1})}})\right|\geq\varepsilon \right)
		\leq \delta.
	\end{equation*}
	Therefore, we obtain that
	\begin{equation}
		\label{thm4Eq7}% \nonumber % Remove numbering (before each equation)
		\pr\left(1-\frac{1}{\ell}\sum_{j=1}^{\ell}\xi_{\tau^{(i_{k-1},j)}} \leq 1+\varepsilon-\E\left(\xi_{\tau^{(i_{k-1})}} \right)\right)
		\geq 1-\delta.
	\end{equation}
	Combining \cref{thm4Eq6} and \cref{thm4Eq7}, with probability at least $ 1-\delta$, we have
	\begin{equation}
		\label{thm4Eq8}%\nonumber % Remove numbering (before each equation)
		\frac{1}{\ell}\sum_{j=1}^{\ell}\E_k\|x^{(k+1,j)}-x_{LS}\|_{A^{\T}A}^2
		\leq \left(1+\varepsilon-\E\left(\xi_{\tau^{(i_{k-1})}} \right)\right) \left(\frac{1}{\ell}\sum_{j=1}^{\ell}\|x^{(k,j)}-x_{LS}\|_{A^{\T}A}^2\right).
	\end{equation}
	Since all batch-samplings $\tau^{(i_{t})}$ ($t=0,\ldots,k-1$) are independent, and follow the same distribution $\bbb{P}$ as $\tau$, replacing $\E[\xi_{\tau^{(i_{k-1})}}]$ by $\E[\xi_{\tau}]$, taking full expectation of \cref{thm4Eq8} and leveraging the arbitrariness of $S$, the rate bound \cref{thm4Eq2} is obtained. Moreover, if $\epsilon,\delta\to 0$, then $1-\delta\to 1$ and $\ell\to\infty$, by the strong law of large numbers, \cref{thm4Eq8} almost surely leads to
	\begin{equation*}
		\E_{k-1}\|x^{(k+1)}-x_{LS}\|_{A^{\T}A}^2
		\leq \left(1-\E\left(\xi_{\tau}\right)\right)   \E_{k-1}\|x^{(k)}-x_{LS}\|_{A^{\T}A}^2,
	\end{equation*}
	then taking full expectation of the above estimate and leveraging the arbitrariness of $S$ result in \cref{thm4Eq2_1}.
\end{proof}

\begin{remark}
	\label{remark:covariance_cd}
	We now validate through numerical experiments that the sample covariance condition $\cov\,(\xi_{\tau^{(i_{k-1})}},\|x^{(k)}-x_{LS}\|_{A^{\T}A}^2)\geq 0$ in \cref{convRBSCD_sharper}, is satisfied with high empirical frequency in practical RBSCD iterations. To examine this,  we randomly generate $100$ test points in the ball centered at $x_{LS}$ with radius $r_0=10$, which are used to mimic possible iterates arising in the RBSCD updates. At each test point $x^{(k-1)}$, we draw $100$ independent batch samples $\{\tau^{(i_{k-1},j)}\}_{j=1}^{100}$, compute the corresponding quantities $\xi_{\tau^{(i_{k-1},j)}}$, and carry out a single RBSCD update for each sample to obtain the values of $\|x^{(k,j)}-x_{LS}\|_{A^{\T}A}^2$. We then calculate the sample covariance between the two sample sequences $\{\xi_{\tau^{(i_{k-1},j)}}\}_{j=1}^{100}$ and $\{\|x^{(k,j)}-x_{LS}\|_{A^{\T}A}^2\}_{j=1}^{100}$, and plot the corresponding sample covariance at each test point. As an illustration on a real-data instance, we perform the experiment on a $1000\times120$ test matrix obtained from the first $120$ columns of the sparse matrix \texttt{cage8} in the SuiteSparse Matrix Collection \cite{davis2011university}. We test three batch sizes, namely, $q\in\{10,20,40\}$, and report the results in \cref{fig:covariance_cd}. In all three cases, the sample covariance is positive for the majority of the sampled test points, with this tendency becoming more pronounced as $q$ increases. This provides empirical support for the reasonableness of the sample covariance condition and, consequently, for the applicability of \cref{convRBSCD_sharper}.
	\begin{figure}[tbhp]
		\centering
		\subfloat[]{\includegraphics[width=0.32\linewidth]{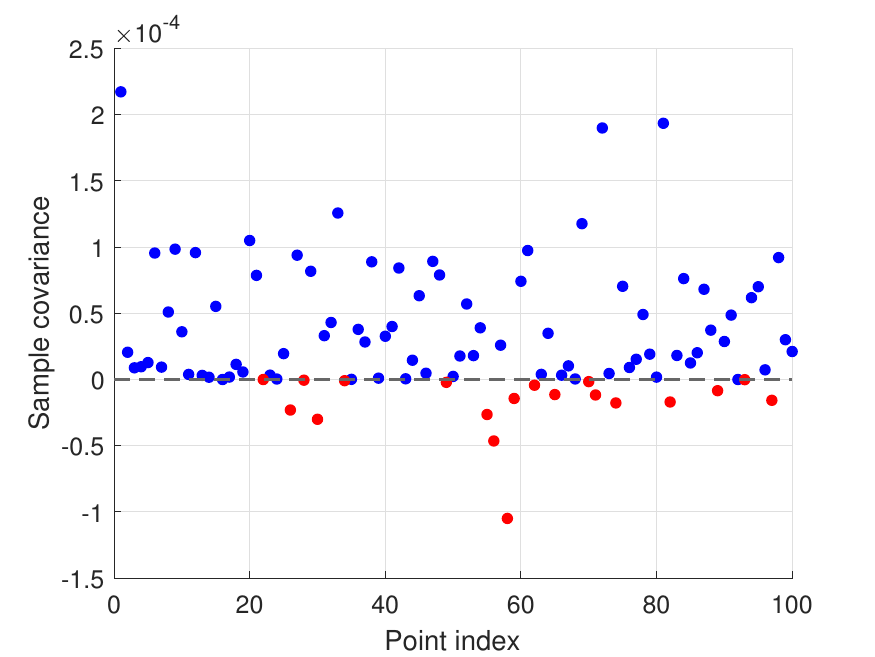}}
		\subfloat[]{\includegraphics[width=0.32\linewidth]{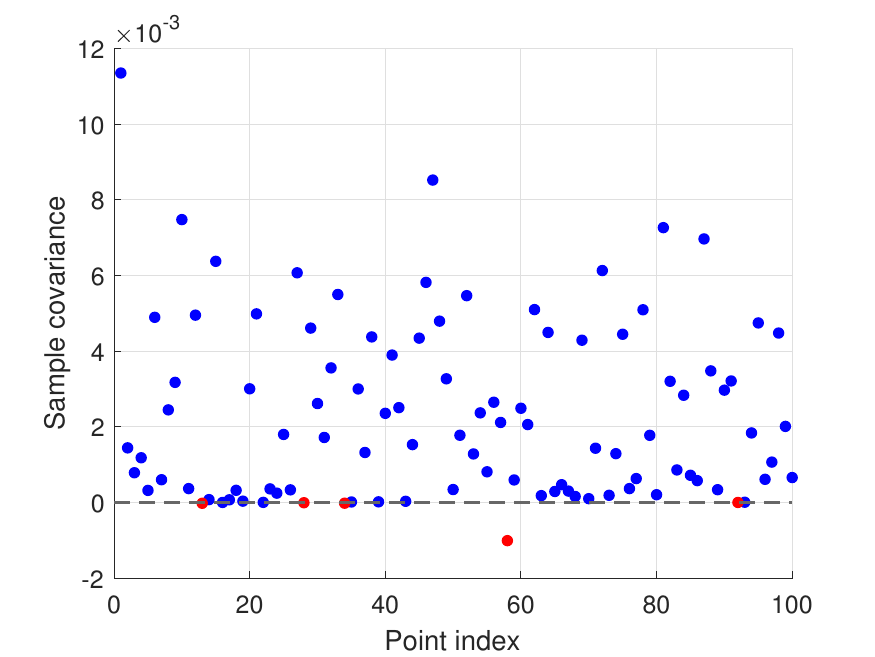}}
		\subfloat[]{\includegraphics[width=0.32\linewidth]{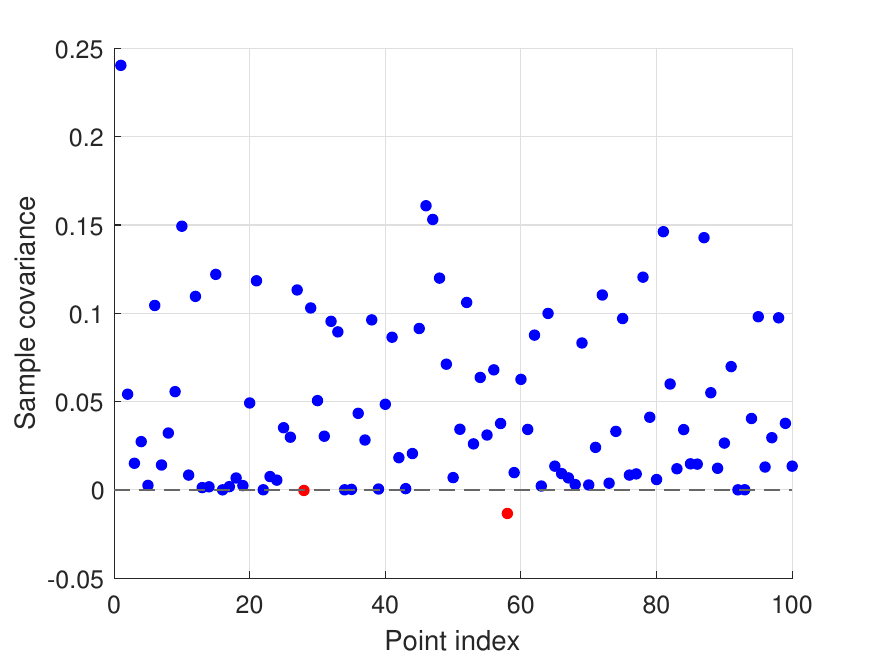}}
		\caption{Sample covariance $\cov\,(\xi_{\tau^{(i_{k-1})}},\|x^{(k)}-x_{LS}\|_{A^{\T}A}^2)$ at sampled test points for a test matrix extracted from \texttt{cage8}. Subfigures (a), (b), and (c) correspond to batch sizes $q=10$, $q=20$, and $q=40$, respectively. Blue points indicate positive sample covariances, while red points indicate negative ones.}
		\label{fig:covariance_cd}
	\end{figure}
\end{remark}

\section{Numerical experiments}\label{sec-experiments}
In this section, we conduct numerical experiments to compare our theoretical convergence rate bounds with existing results in the literature, thereby assessing the sharpness and effectiveness of the proposed upper bounds. Furthermore, we focus on the CT reconstruction problem, where we train a parameterized sampling distribution based on the theoretical framework developed in this paper and subsequently incorporate the learned distribution into the RBSCD iterations. Additional experiments are performed to compare the iteration convergence behavior and the actual computational efficiency of the learned-distribution RBSCD method with several recently proposed efficient coordinate descent methods on CT problems. All experiments are implemented by using MATLAB R2024b on a personal computer with 3.40GHz central processing unit (Intel(R) Core(TM) i7-14700KF), 32 GB memory, and Windows 11 operating system.

\subsection{Numerical verification of theoretical upper bounds}\label{section 5.1}
In this subsection, we numerically compare the convergence rate upper bounds obtained in this work for the batch-sampling RBSCD framework with the classical bound proposed by Needell et al. in \cite{NeedellLAA15,WuNeedell2018RBGS} (denoted by ND15) and the bound provided by Gower et al. in \cite{gower2021adaptive} (denoted by GM21). To ensure a fair comparison, our experiments adopt a batch-sampling setup consistent with the implementation assumptions in ND15 and GM21. While our analysis depends only on the sampling distribution and does not require explicitly constructing all batch matrices, the ND15 and GM21 implementations typically rely on a finite collection of preconstructed blocks.
In particular, ND15 further requires a partition of the column indices into mutually disjoint blocks, so that each block contains distinct coordinates.
%In particular, under the GM21 setup, using an overly large number of candidate batches is not advisable because it substantially increases the per-iteration overhead.
Moreover, since adaptive batch-sampling is beyond the scope of this work, we restrict our comparison to the non-adaptive setting.

Since least-squares problems typically arise from overdetermined linear systems, we restrict our experiments to the tall matrix regime.
Specifically, we fix $n=120$ and take $m\in\{500,600,700,800,900,1000\}$, the test instances focus on data matrices with
%multi-scale structure and ill-conditioning,
ill-conditioning structure,
where the original matrices $A\in\reals^{m\times n}$ are constructed from both (i) Gaussian random matrices generated by MATLAB function \texttt{randn} and (ii) full column rank matrices extracted from real datasets in the SuiteSparse Matrix Collection \cite{davis2011university}. To generate ill-conditioned test matrices, we modify a subblock of $A$ so that its singular values exhibit a pronounced decay. More precisely, we first compute the smallest singular value $\sigma_{\min}(A)$ of the original matrix $A$, and draw Haar-distributed orthogonal matrices $U\in\mathbb{R}^{m\times m}$ and $V\in\mathbb{R}^{q\times q}$. Next, we form a generalized diagonal matrix $\Sigma\in\mathbb{R}^{m\times q}$ whose nonzero diagonal entries decrease monotonically along the main diagonal, with $\sigma_i=\rho\sigma_{\min}(A)-(i-1)h$, where $\rho=0.2$ and $h=0.001$, so that the largest singular value of the constructed block equals $\rho\sigma_{\min}(A)$. Finally, we set $B=U\Sigma V^{\top}$ and replace one subblock of $A$ by $B$ to obtain an ill-conditioned variant of the original test matrix. To eliminate the influence of noise when comparing empirical average convergence rates with theoretical convergence factors, we consider consistent linear systems by generating $x_\star=\texttt{randn}(n,1)$ and setting the right-hand side as $b=Ax_\star$. In the RBSCD iterations, the columns of
$A$ are uniformly partitioned into $s=n/q$ blocks of size $q$, that is, $A=[A_{\tau^{(1)}},A_{\tau^{(2)}},\ldots,A_{\tau^{(s)}}],$ and the blocks are sampled according to a fixed distribution $p_j \propto \left\| A_{\tau^{(j)}} \right\|^2$,  as suggested in \cite{gower2015randomized}. In all experiments, we initialize $x^{(0)}=0$, and the relative residual (RRes) is defined as $\mathrm{RRes}=\left\|b-Ax\right\|_2 / \left\|b\right\|_2$.
The algorithm terminates when $\mathrm{RRes}<10^{-10}$ or after 10000 iterations. The numerical experiments examine the convergence rate bounds provided by \cref{convRBSCD,convRBSCD_sharper}, which are referred to as `Ours-I' and `Ours-II', respectively. Each figure presents the empirical average convergence curve over 30 independent runs as a black dashed line. The shaded region represents the interquartile range of the empirical convergence curves, namely, the interval from the 25th percentile to the 75th percentile.

In all experiments, we adopt a fixed $S=\diag\left(\frac{1}{\|A^{(1)}\|_2},\ldots,\frac{1}{\|A^{(n)}\|_2}\right)$,
where $A^{(j)}$ denotes the $j$th column of $A$. We emphasize that this is only a specific choice of $S$, whereas a sharper bound could be obtained by optimizing over $S$. For each sampled batch $\tau$, the matrix $W_{\tau}$ is constructed by setting $W_{\tau}=Q_AY_{\tau}$, where $Q_A\in\reals^{m\times n}$ is an orthonormal basis matrix of $\range(A)$ obtained from a thin QR factorization of $A$, and the columns of $Y_{\tau}$ form an orthonormal basis of $\zspace(A_{\tau}^{\top}Q_A)$ computed by MATLAB function \texttt{null}, so that the columns of $W_{\tau}$ form an orthonormal basis of $\mathcal{W}_{\tau}=\range(A)\cap\zspace(A_{\tau}^{\top})$. In addition, the diagonal matrices $P_1,\ldots,P_q$ defined in \cref{marDistribMatrix} take the form $P_i=\diag(p_1E_{ii},\,p_2E_{ii},\,\ldots,\,p_sE_{ii}),$
where $E_{ii}\in\mathbb{R}^{q\times q}$ denotes the diagonal matrix whose $(i,i)$ entry is $1$ and all other entries are $0$.

\cref{fig:my_label4,fig:my_label5,fig:my_label6} present the comparisons between the theoretical convergence rates and the empirical average convergence rates for ill-conditioned matrices constructed from both Gaussian random matrices and the sparse matrices \texttt{cage8} and \texttt{bibd\_81\_2} in \cite{davis2011university}. The results show that ‘Ours-II' is consistently the tightest among all considered upper bounds and tracks the empirical average rate more closely than both GM21 and ND15, with the gap becoming more visible as the block size $q$ increases. By contrast, `Ours-I' remains numerically close to GM21 while consistently providing a sharper bound than ND15.

\begin{figure}[!tbhp]
	\centering
	\subfloat[]{\includegraphics[width=0.32\linewidth]{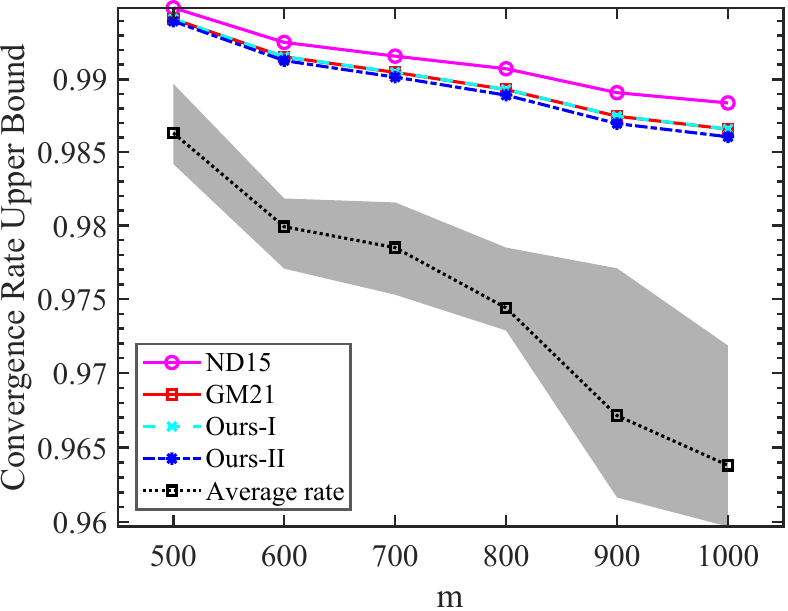}}\hfill
	\subfloat[]{\includegraphics[width=0.32\linewidth]{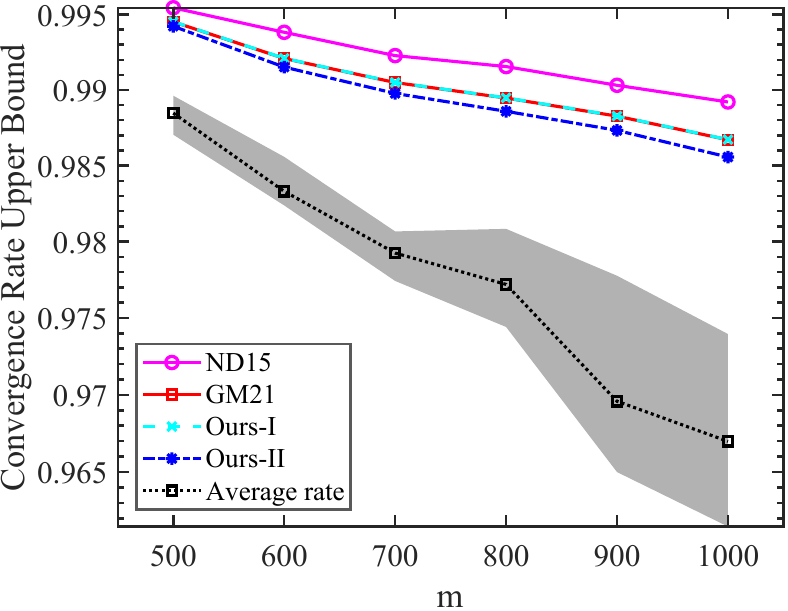}}\hfill
	\subfloat[]{\includegraphics[width=0.32\linewidth]{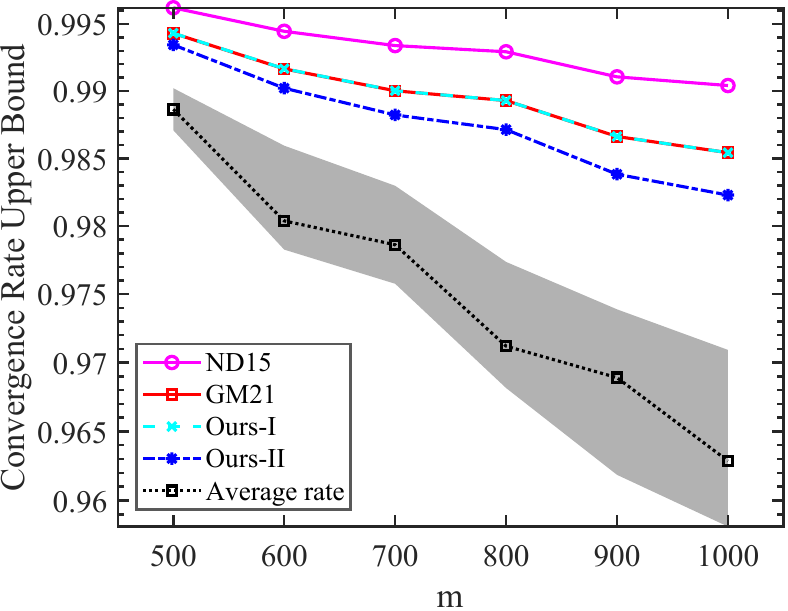}}\hfill
	\caption{Theoretical upper bounds on the convergence rate for ill-conditioned test matrices constructed from Gaussian random matrices. Results are shown for block sizes $q=10$, $q=20$, and $q=40$ in (a), (b), and (c), respectively.}
	\label{fig:my_label4}
\end{figure}

\begin{figure}[!t]
	\centering
	\subfloat[]{\includegraphics[width=0.32\linewidth]{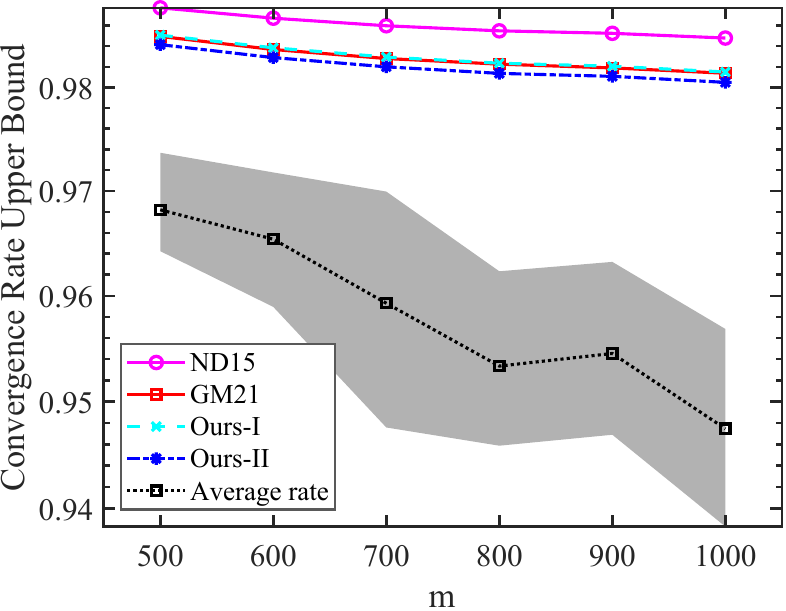}}\hfill
	\subfloat[]{\includegraphics[width=0.32\linewidth]{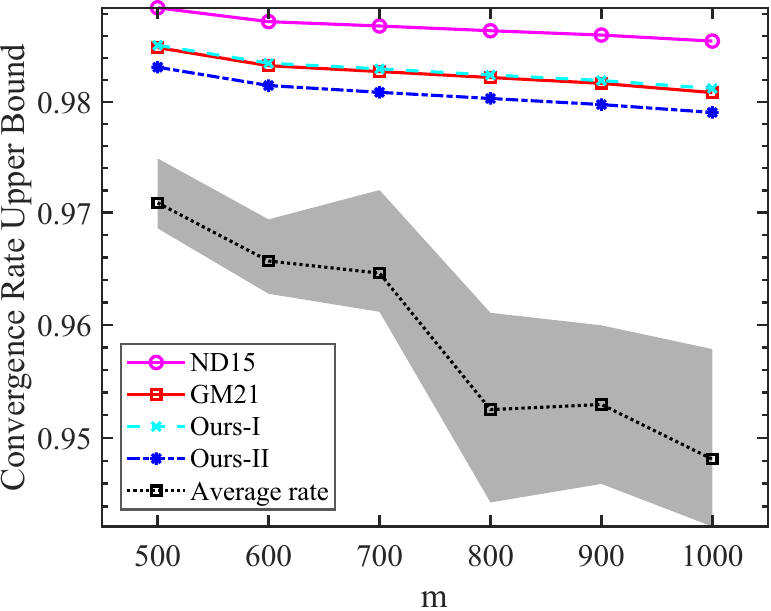}}\hfill
	\subfloat[]{\includegraphics[width=0.32\linewidth]{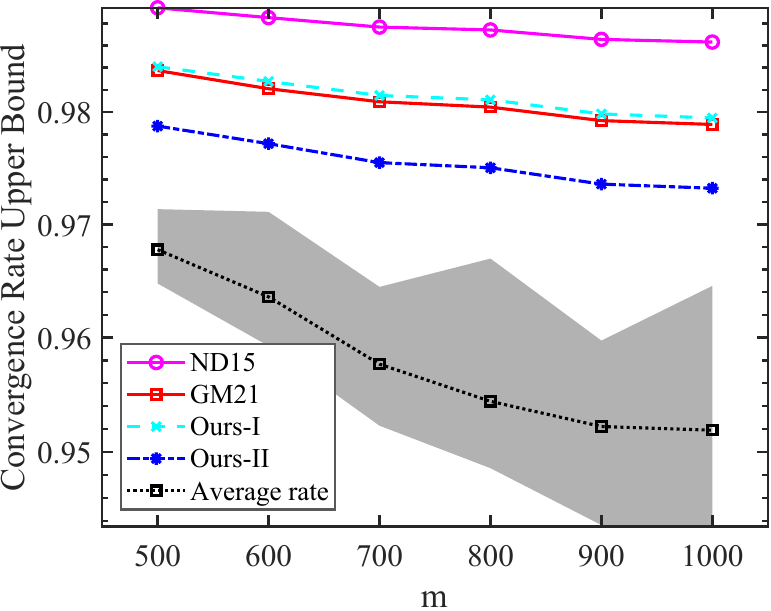}}\hfill
	\caption{Theoretical upper bounds on the convergence rate for ill-conditioned test matrices constructed from an $m \times 120$ submatrix extracted from \texttt{cage8}. Results are shown for block sizes $q=10$, $q=20$, and $q=40$ in (a), (b), and (c), respectively.}
	\label{fig:my_label5}
\end{figure}

\begin{figure}[!t]
	\centering
	\subfloat[]{\includegraphics[width=0.32\linewidth]{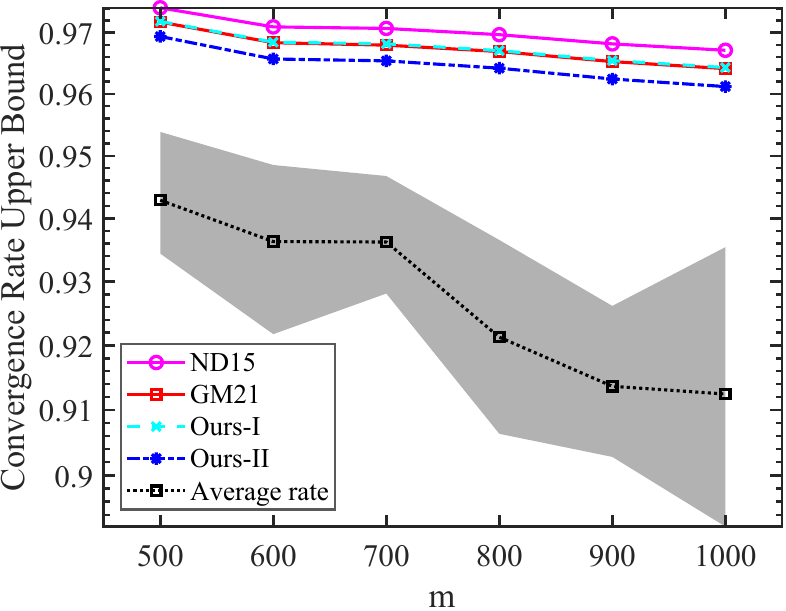}}\hfill
	\subfloat[]{\includegraphics[width=0.32\linewidth]{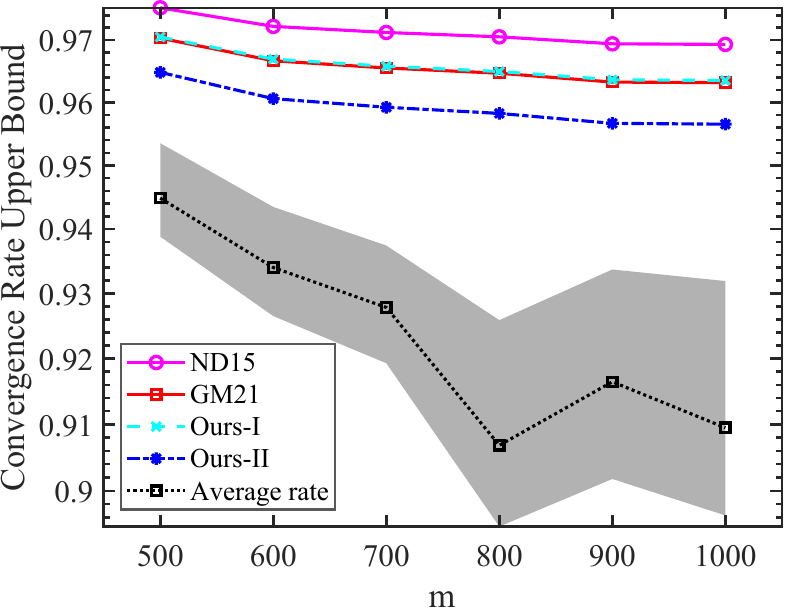}}\hfill
	\subfloat[]{\includegraphics[width=0.32\linewidth]{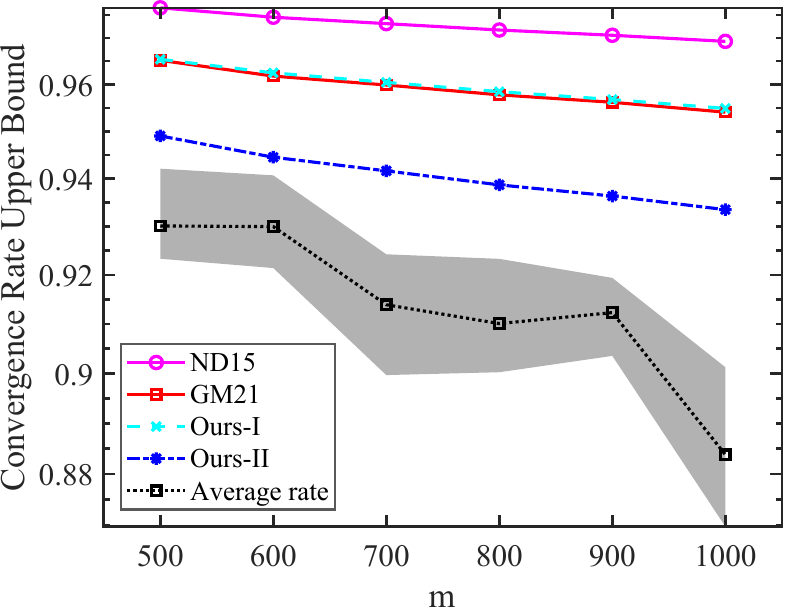}}\hfill
	\caption{Theoretical upper bounds on the convergence rate for ill-conditioned test matrices constructed from an $m \times 120$ submatrix extracted from \texttt{bibd\_81\_2}. Results are shown for block sizes $q=10$, $q=20$, and $q=40$ in (a), (b), and (c), respectively.}
	\label{fig:my_label6}
\end{figure}

The aforementioned experimental results can be explained as follows. For the sake of clarity, we let
\[
   H =  A_{\tau} A_{\tau}^{\dagger}.
\]
Both SAP and RBSCD provides worst-case lower bounds $\beta_j$ as follows
\begin{align*}
  \langle \mathbb{E}(H) Ae^{(k)}, Ae^{(k)} \rangle & \ge \left( \min_{y\in\mathrm{range}(A^{\top})} \frac{\langle \mathbb{E}(H) Ay, Ay \rangle}{\langle Ay, Ay \rangle} \right) \langle Ae^{(k)}, Ae^{(k)} \rangle \\
    & \ge \beta_{j} \langle Ae^{(k)}, Ae^{(k)} \rangle,\ \forall j=\text{\footnotesize SAP, GM21, RBSCD}.
\end{align*}
where $e^{(k)}=x^{(k)}-x_{LS}$ is the error vector at the $k$th iterate $x^{(k)}$. 

Specifically, for the SAP framework, the analyses in \cite{gower2015randomized,richtarik2020} yielded
\[
    \beta_{\text{SAP}} = \lambda_{\min} ((A^{\top}A)^{-1} A^{\top} \mathbb{E}(H) A),
\]
which led to the following worst-case convergence rate bound
\[
\rho_{\text{SAP}}
=
1-\beta_{\text{SAP}}.
\]
This bound was subsequently improved in \cite{gower2021adaptive} to
\[
\rho_{\text{GM21}}
=
1-\beta_{\text{GM21}},
\ \beta_{\text{GM21}}
\geq
\beta_{\text{SAP}}.
\]
These two bounds are very close, i.e.,
\[
    \rho_{\text{GM21}} \approx \rho_{\text{SAP}}
\]
with \(\rho_{\text{GM21}}\) being slightly smaller. Therefore, the convergence rate bound $\rho_{\text{GM21}}$ is used as the principal SAP-based reference in our numerical comparisons. 

In the reformulated RBSCD framework, our analysis yields
\[
    \beta_{\text{RBSCD}} = \xi_{\tau}  = \lambda_{\min} ( W_{\tau}^{\top} A M^2 A^{\top}W_{\tau} ),
\]
which leads to the following two convergence rate bounds: 
\begin{itemize}
	\item the worst-case bound in \cref{convRBSCD,convRBSCD-noise}
	\[
	\rho_{\text{Ours-I}} = 1-\min_{\tau\sim\mathbf{P}} \beta_{\text{RBSCD}},
	\]
	which is based on the minimum value of the worst-case quantity $\beta_{\text{RBSCD}}$ evaluated over all possible restricted-subspace;	
	\item the non-worst-case bound in \cref{convRBSCD_sharper}
	\[
	\rho_{\text{Ours-II}} = 1-\mathbb{E}(\beta_{\text{RBSCD}}),
	\]
	which is based on the expected value of the worst-case quantity $\beta_{\text{RBSCD}}$ evaluated over all possible restricted-subspace.
\end{itemize}

As shown above, despite their distinct formulations, $\rho_{\text{SAP}}$, $\rho_{\text{GM21}}$, and $\rho_{\text{Ours-I}}$ all represent worst-case convergence rate bounds. As a result, the discrepancy between $\rho_{\text{GM21}}$ and $\rho_{\text{Ours-I}}$ is hardly observable in the numerical experiments. Furthermore, the improvement of $\rho_{\text{Ours-II}}$ arises from expressing the bound with the expectation. Intuitively, randomized iterations sample batches with varying effectiveness rather than repeatedly sampling the least effective batch. By weighting the iteration progress of individual batches based on their respective sampling probabilities, the convergence rate bound $\rho_{\text{Ours-II}}$ provides a more faithful characterization of the average iteration progress and produces a tighter bound, which in turn better captures the empirical convergence behavior.

Each block coordinate descent iteration accesses at most \(q\) columns of \(A\). Therefore, \(n/q\) iterations process at most \(n\) column accesses in total, corresponding to one data-equivalent epoch. This provides a natural workload-normalized interpretation of the convergence rate bounds. Accordingly, we consider the following one-epoch convergence rate bounds:
\[
\rho_{\text{GM21}}^{\,n/q}
=
\left[1-\beta_{\text{GM21}}\right]^{n/q},
\qquad
\rho_{\text{Ours-II}}^{\,n/q}
=
\left[
1-\mathbb{E}(\beta_{\text{RBSCD}})
\right]^{n/q}.
\]
Using the convergence rate bound data from \cref{fig:my_label6}, the following \Cref{tab:1} reports the corresponding one-epoch bounds for \(q\in\{10,20,40\}\). The smaller value in each pair is shown in bold. 

\begin{table*}[htbp]
	\centering
	\caption{One-epoch convergence rate bounds for varying block size \(q\) and matrix row number \(m\).}
	\label{tab:1}
	\renewcommand{\arraystretch}{1.15}
	\setlength{\tabcolsep}{7pt}
	\begin{tabular}{cc*{6}{c}}
		\toprule
		\multirow{2}{*}{$q$}
		& \multirow{2}{*}{Bound}
		& \multicolumn{6}{c}{Matrix row number $m$} \\
		\cmidrule(lr){3-8}
		& & 500 & 600 & 700 & 800 & 900 & 1000 \\
		\midrule
		
		\multirow{2}{*}{10}
		& $\rho_{\text{GM21}}^{\,n/q}$
		& 0.708920 & 0.679806 & 0.676445
		& 0.667713 & 0.654239 & 0.644953 \\
		& $\rho_{\text{Ours-II}}^{\,n/q}$
		& \textbf{0.688445} & \textbf{0.657713} & \textbf{0.655217}
		& \textbf{0.645534} & \textbf{0.631505} & \textbf{0.621891} \\
		\midrule
		
		\multirow{2}{*}{20}
		& $\rho_{\text{GM21}}^{\,n/q}$
		& 0.834619 & 0.816025 & 0.810164
		& 0.805945 & 0.798853 & 0.798384 \\
		& $\rho_{\text{Ours-II}}^{\,n/q}$
		& \textbf{0.806668} & \textbf{0.785759} & \textbf{0.779049}
		& \textbf{0.774277} & \textbf{0.766642} & \textbf{0.766005} \\
		\midrule
		
		\multirow{2}{*}{40}
		& $\rho_{\text{GM21}}^{\,n/q}$
		& 0.899089 & 0.889665 & 0.884357
		& 0.878553 & 0.874186 & 0.868399 \\
		& $\rho_{\text{Ours-II}}^{\,n/q}$
		& \textbf{0.854684} & \textbf{0.842584} & \textbf{0.834913}
		& \textbf{0.827119} & \textbf{0.820949} & \textbf{0.813509} \\
		\bottomrule
	\end{tabular}
\end{table*}

As shown in \Cref{tab:1}, evaluating the bounds over one epoch makes the advantage of \(\rho_{\text{Ours-II}}\) over \(\rho_{\text{GM21}}\) substantially more pronounced for all tested \(q\) and \(m\). Although their per-iteration difference may appear small because both bounds are close to one, the cumulative improvement over one data-equivalent epoch is clearly non-negligible. This comparison further demonstrates that our sharper bound is not merely an incremental improvement.

Although `Ours-II' is the tightest of the bounds considered above, \cref{fig:my_label4,fig:my_label5,fig:my_label6} still show a visible gap between this bound and the empirical average rate. A principal source of this remaining gap is that the scaling matrix $S$ used above is fixed and generally nonoptimal. As discussed in \cref{rem:optimization}, obtaining an optimizer $S$ requires solving a challenging optimization problem. We therefore conduct an additional small-scale experiment, for which a coarse-to-fine multilevel grid search over $S$ is computationally feasible, to provide preliminary numerical evidence, rather than a rigorous proof, of how much this gap can be reduced by optimizing $S$.

We set $m=8$, $n=6$, and the block size $q=3$, then generate \(A\in\mathbb{R}^{8\times6}\) with independent standard
Gaussian entries. As in the preceding experiments, each block is sampled with probability proportional to the squared spectral norm of the corresponding block submatrix. The coordinate indices are partitioned into $\tau^{(1)}=\{1,2,3\}$ and $\tau^{(2)}=\{4,5,6\}$. The fixed baseline uses the same column-normalized scaling as above, namely, $S_{\mathrm{fixed}}=\diag (1/\|A^{(1)}\|_2,\ldots,1/\|A^{(6)}\|_2 )$.
%
%For $S=\diag(s_1,\ldots,s_6)$ and $\ell\in\tau^{(j)}$, $j=1,2$, the corresponding diagonal entry of $M$ is
%\begin{equation*}
%	[M]_{\ell\ell}
%	=\frac{\sqrt{p_j}\,|s_\ell|}
%	{\left\|A_{\tau^{(j)}}S_{\tau^{(j)}}\right\|_2}.
%\end{equation*}
%Consequently, multiplying all diagonal entries of $S_{\tau^{(j)}}$ by the same nonzero constant leaves $M$ unchanged and therefore does not affect the resulting bound. Thus, only the relative values of these entries within each block matter.
%Exploiting this property,
%
To simplify the search for $S$, we parameterize it as $S(z)=\diag\left(e^{z_1},e^{z_2},e^{z_3},e^{z_4},e^{z_5},e^{z_6}\right)$ with $z\in\reals^6$, and minimize the convergence factor `Ours-II' using a three-level coarse-to-fine grid search. The first level searches $[-2,2]^6$ with spacing $0.25$. The second level searches a box of half-width $0.25$, centered at the best point from the first level, with spacing $0.025$. The third level uses half-width $0.025$ and spacing $0.0025$ around the best point from the second level. We denote the convergence rate bound `Ours-II' with respect to the fixed and grid-optimized $S$ by $\rho_{\mathrm{fixed}}$ and $\rho_{\mathrm{opt}}$, respectively, and denote the empirical convergence factor by $\rho_{\mathrm{emp}}$. The error bar in \cref{fig:S_optimization}(a) indicates the empirical 95\% confidence interval (CI) for $\rho_{\mathrm{emp}}$. 

\begin{figure}[!tbhp]
	\centering
	\subfloat[Per-iteration convergence factors.]{\includegraphics[width=0.48\linewidth]{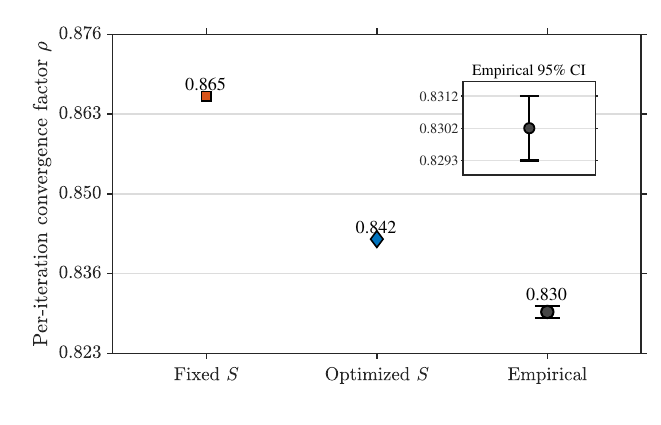}}
	\hfill
	\subfloat[Percentage relative to empirical contraction.]{\includegraphics[width=0.48\linewidth]{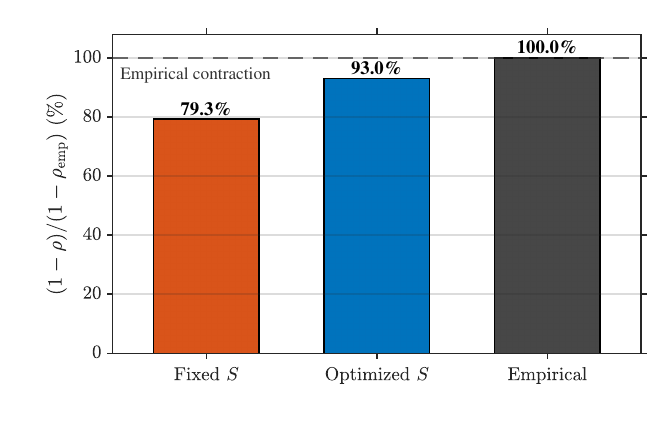}}
	\caption{Effect of optimizing $S$ on `Ours-II'. (a) Fixed-$S$ and grid-optimized convergence factors compared with the empirical convergence factor; the error bar indicates the empirical 95\% CI for $\rho_{\mathrm{emp}}$. (b) Percentage relative to empirical contraction, $(1-\rho)/(1-\rho_{\mathrm{emp}})$, with the empirical value as the 100\% reference.}
	\label{fig:S_optimization}
\end{figure}

As shown in \cref{fig:S_optimization}(a), the optimized scaling matrix $S$ significantly narrows the gap between the theoretical convergence factor bound and its empirical contraction factor. The normalized comparison in \cref{fig:S_optimization}(b) clearly quantifies the above gap reduction, in which $1-\rho$ and $1-\rho_{\mathrm{emp}}$ refer to the step sizes of theoretical and empirical contraction for one iteration respectively. This experiment provides a direct visual illustration of the tightness attainable by our bound through a suitable choice of $S$. Developing a theoretical or computational method to find a globally optimal $S$, thereby obtaining as sharp bounds as possible, remains an interesting direction for future research.

\subsection{Construction of learning-based optimal distributions for RBSCD}\label{section 5.2}
Based on the preceding theoretical analysis in \cref{rem:learning}, computing an optimal or near-optimal sampling distribution is often infeasible in practice and may incur substantial computational cost. In this subsection, we consider the reuse strategy of learned optimal distributions in specific scenarios. Specifically, we adopt a data-driven approach to train the sampling distribution based on the RBSCD framework. Subsequently, the learned distribution can be directly used in the RBSCD iterations, leading to faster convergence and reduced runtime.

Taking CT reconstruction as an application, we learn the batch-sampling distribution from training instances, and reuse the learned optimal distribution in RBSCD iterations for reconstructing new unknown objectives. In CT scanning, an X-ray source and a detector system rotate around the imaging domain and collect projection data at discrete angles $\{\theta_k\}_{k=1}^{N_\theta}$. As the X-ray traverses the object, its intensity is attenuated, with the attenuation depending on the linear attenuation coefficients along the ray path and the corresponding path lengths. Let $\mu(s)$ denote the linear attenuation coefficient at position $s$ along a ray, and let $b$ denote the total attenuation for that ray. The measurement is modeled by the line integral $b=\int_{L}\mu(s)\,\mathrm{d}s$ along the ray path $L$.

In a discrete model, the cross-section to be reconstructed is partitioned into an $N\times N$ grid of equal-size pixels, and $\mu(s)$ is assumed constant on each pixel. Let $x\in\mathbb{R}^{n}$ with $n=N^2$ denote the vector of pixel-wise attenuation coefficients, where $x_j$ corresponds to the $j$th pixel. When reshaped according to the $N\times N$ grid, $x$ represents the reconstructed grayscale CT image, whose pixel values are determined by the corresponding attenuation coefficients. For the $i$th ray, the measurement is approximated by
\begin{equation*}
	b_i=\sum_{j=1}^{N^2} a_{ij}x_j,
\end{equation*}
where $a_{ij}$ is the intersection length of the $i$th ray with the $j$th pixel, and $a_{ij}=0$ if the ray does not pass through that pixel. For each angle $\theta_k$, the detector records the projection intensities of $p$ parallel rays. Aggregating the measurements over all $N_\theta$ angles yields $m=pN_\theta$ equations, which can be expressed as the linear system $Ax=b$, where $A=[a_{ij}]\in\mathbb{R}^{m\times n}$ and $b=[b_1,b_2,\ldots,b_m]^{\top}\in\mathbb{R}^m$. Under this simplified CT model, the system matrix $A$ is determined by the scanner geometry and acquisition protocol, including the grid resolution $N$, the projection angle vector $\theta$, and the number of parallel rays $p$. For a specified examination protocol (e.g., a standardized head CT protocol), these nominal acquisition parameters are typically fixed values or selected from a predefined set, thereby corresponding to a specific system matrix $A$. By contrast, different scanned objects generally have different attenuation-coefficient vectors $x$ and hence different projection vectors $b$. Therefore, repeated reconstructions performed under the same nominal scanner geometry and acquisition protocol constitute multiple-right-hand-side/repeated-solve problems, since they share the same nominal system matrix $A$, while their right-hand-sides and solutions may differ substantially.

To construct the training and validation datasets for learning the sampling distribution, we generate two-dimensional parallel-beam CT test problems using the \texttt{paralleltomo} routine from the AIRToolsII package \cite{Hansen2018}. In total, we generate $\ell=11$ instances, where the first $(\ell-1)$ instances are used for training and the last instance is used to evaluate the computational performance of the learned distribution in iterative solves. Specifically, we call $[A,b,x]=\texttt{paralleltomo}(N,\theta,p)$ and set $N=20$, $N_\theta=180$, and $p=10$, which yields a test problem with $m=1800$ and $n=400$. In practice, the actual projection angles may deviate slightly from their nominal values because of calibration uncertainty and minor mechanical positioning errors, and the measured projection vector is typically contaminated by noise. Accordingly, we start from the baseline angle vector $\theta=0\!:\!1\!:\!179$ and generate perturbed angles $\theta^{(i)}=\theta+\epsilon^{(i)}$ for $i=1,\ldots,\ell$, where $\epsilon^{(i)}\sim\mathcal{N}(0,\sigma_\theta^2 I_{N_\theta})$ are independent Gaussian noise with $\sigma_\theta=10^{-1}$. We then use $(N,\theta^{(i)},p)$ as the input to \texttt{paralleltomo} to generate the corresponding CT instances $\{A^{(i)},b^{(i)},x^{(i)}\}_{i=1}^{\ell}$. Finally, to model detector measurement noise, we add independent Gaussian noise $w^{(i)}\sim\mathcal{N}(0,\sigma_b^2 I_m)$ with $\sigma_b=10^{-3}$ to the projection vector $b^{(i)}$, forming the noisy right-hand side $\bar b^{(i)}=A^{(i)}x^{(i)}+w^{(i)}$. Random perturbations in the angles and projection vectors serve as a simplified model for random interference factors in practical scenarios, thereby corresponding to a perturbed system matrix. Training over these random perturbations is intended to make the learned distribution less sensitive to small deviations from the nominal scanner configuration. The resulting instances constitute the CT dataset used for training and validation.

Since a smaller convergence factor generally leads to faster practical convergence of RBSCD, we adopt the theoretical convergence factor as the training objective and aim to learn a sampling distribution that makes the factor as small as possible. This theoretical factor depends on the system matrix $A$ and the sampling distribution $\bbb{P}$. Thus, the learned distribution is associated with the scanner configuration and acquisition protocol. In the training procedure, we parameterize the joint distribution $\bbb{P}$ by $logits$ and set $\bbb{P}=\texttt{softmax}(logits)$, where \texttt{softmax} is a classical activation function in machine learning that ensures $\bbb{P}$ lies on the probability simplex. Moreover, we use \texttt{Adam} as the optimizer with learning rate $10^{-1}$. \cref{fig:my_label7} presents the evolution of the loss value over training epochs for different batch sizes $q$. To demonstrate the improvement achieved by the learned distribution in terms of empirical convergence rate and computational efficiency of RBSCD on CT problems, we consider three sampling strategies and denote by $\texttt{P\_uniform}$, $\texttt{P\_normal}$, and $\texttt{P\_train}$ the RBSCD methods based on uniform sampling, sampling proportional to the squared block spectral norms, and sampling according to the learned distribution, respectively. \cref{fig:my_label8} compares these three strategies by plotting the relative residual error curves over a fixed budget of 10000 iterations. To reduce the randomness in individual trials, each curve represents the median over 30 independent trials.

\cref{fig:my_label7} shows that the training loss decreases over a number of epochs and then levels off, which indicates that the learning procedure stabilizes and generates a sampling distribution with a smaller theoretical convergence factor. This validates the effectiveness of the proposed learning-based strategy for optimizing the sampling distribution. \cref{fig:my_label8} presents the iterative behavior on the test CT instance. Under the same iteration budget, the learned-distribution RBSCD method achieves smaller relative residual errors than the uniform and squared-norm-proportional sampling strategies, which demonstrates faster empirical convergence. These results support an offline learning and online deployment workflow. For a prescribed scanner configuration and acquisition protocol, one can train a sampling distribution offline, store it together with the corresponding protocol parameters, and reuse it for subsequent unseen projection data acquired under the same or a similar configuration. Separate distributions may likewise be trained and stored for different standardized protocols, so that the corresponding distribution can be selected directly when a particular protocol is used. In this way, the offline training cost can be amortized across a large volume of shipped CT devices and a massive number of subsequent CT reconstruction tasks. Overall, the experiments suggest that the learned distribution captures structural features of CT problems and can effectively accelerate the convergence of RBSCD.

\begin{figure}[tbhp]
	\centering
	\subfloat[]{\includegraphics[width=0.32\linewidth]{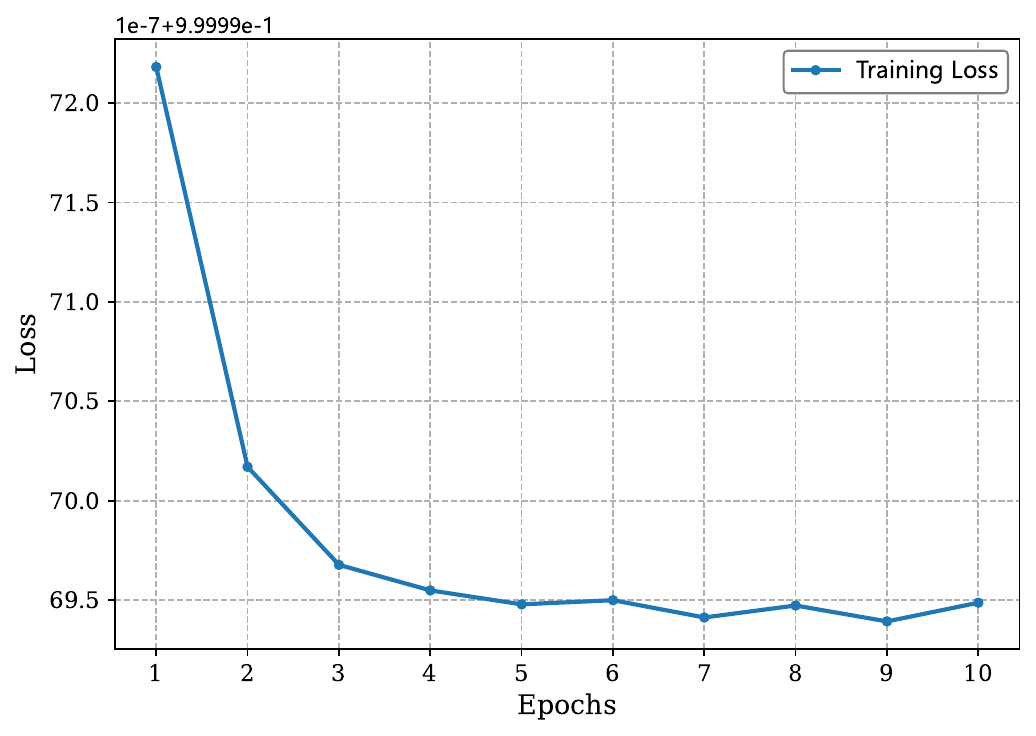}}
	\subfloat[]{\includegraphics[width=0.32\linewidth]{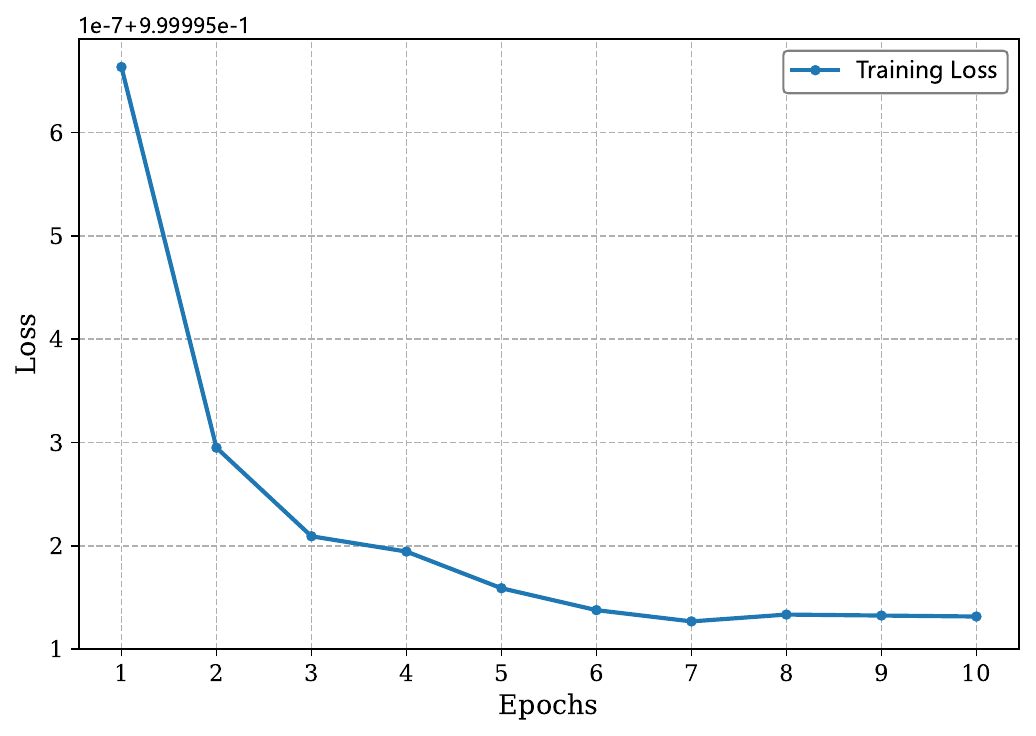}}
	\subfloat[]{\includegraphics[width=0.32\linewidth]{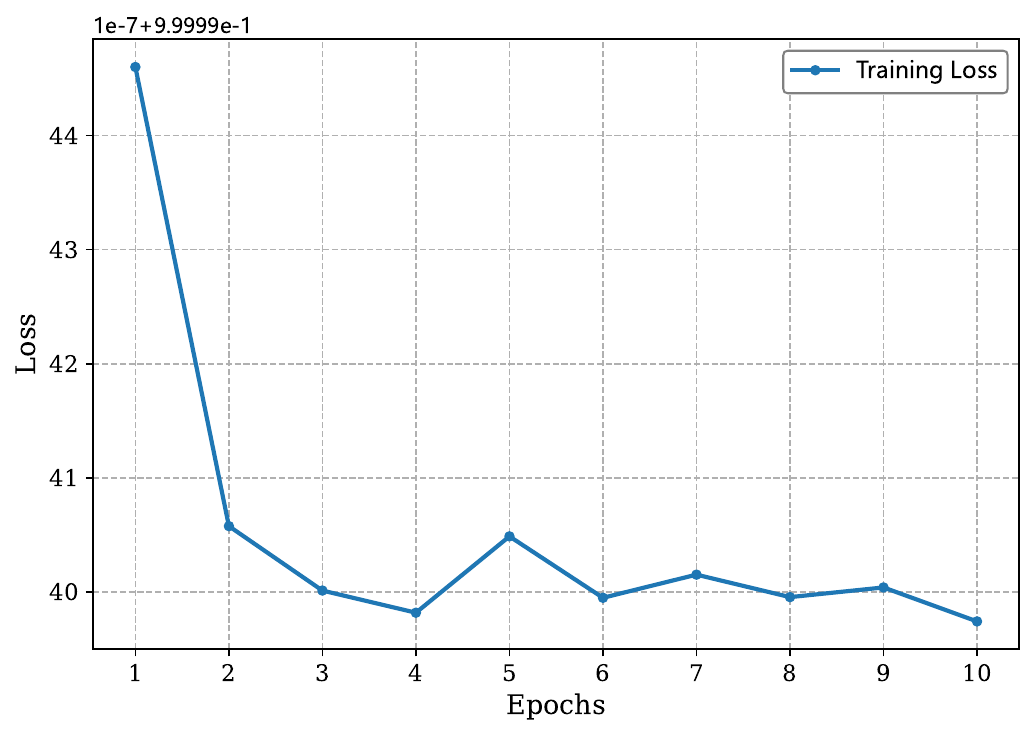}}
	\caption{ Training loss curves over epochs for learning the sampling distributions. Subfigures (a), (b) and (c) correspond to block sizes $q=10$, $q=20$ and $q=50$ respectively.}
	\label{fig:my_label7}
\end{figure}

\begin{figure}[tbhp]
	\centering
	\subfloat[]{\includegraphics[width=0.32\linewidth]{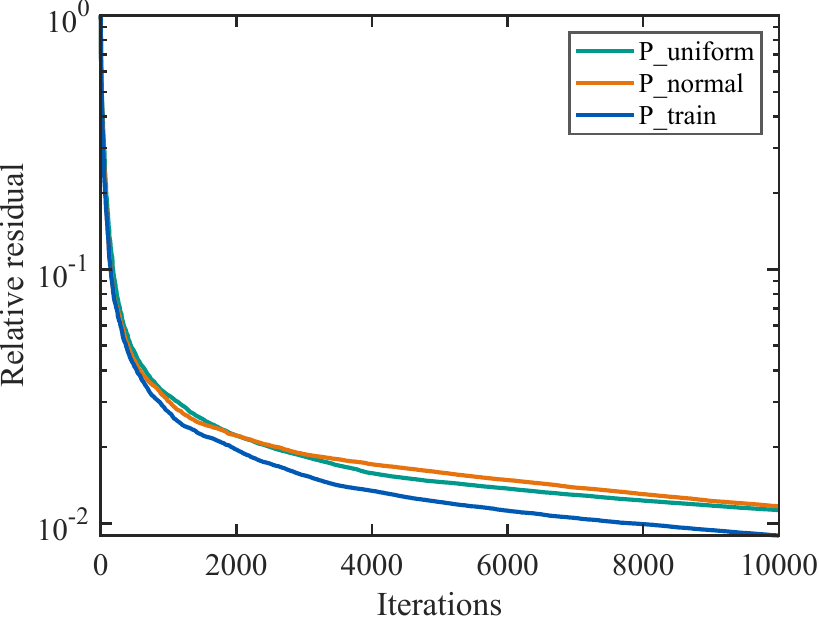}}
	\subfloat[]{\includegraphics[width=0.32\linewidth]{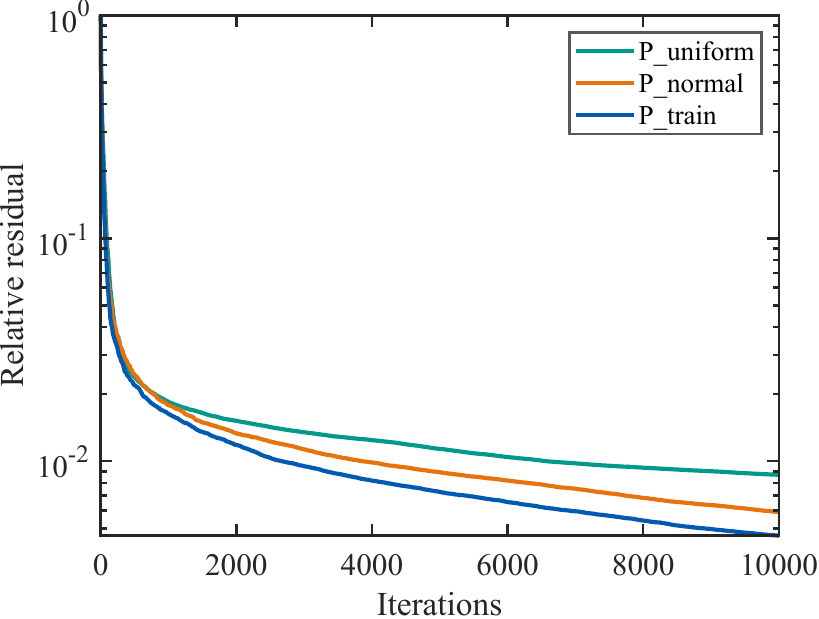}}
	\subfloat[]{\includegraphics[width=0.32\linewidth]{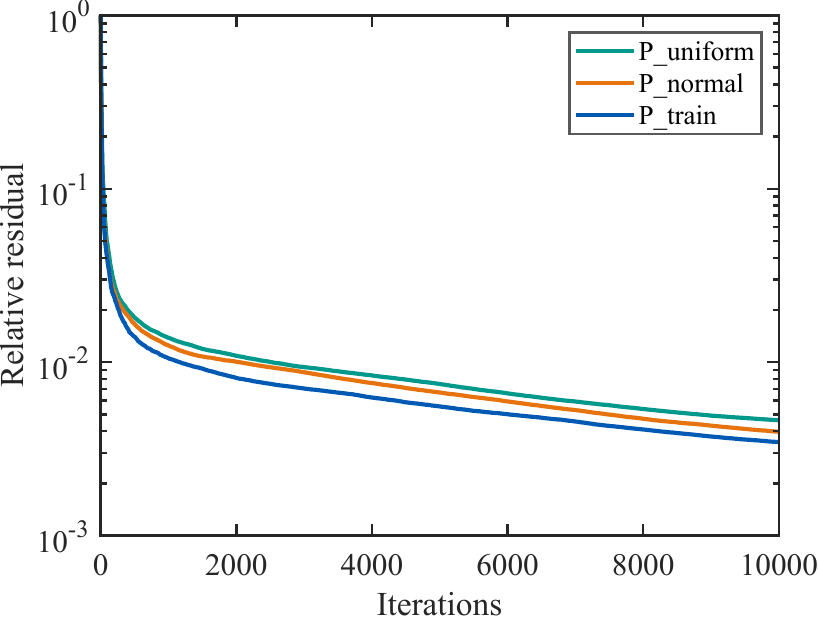}}
	\caption{Iteration error curves of RBSCD for CT reconstruction under different sampling distributions. Subfigures (a), (b) and (c) correspond to block sizes $q=10$, $q=20$ and $q=50$ respectively.}
	\label{fig:my_label8}
\end{figure}

\subsection{Comparison with other coordinate descent methods}\label{section 5.3}
We further evaluate the computational efficiency of the learned optimal distribution reused RBSCD method on CT problems by comparing it against several recently proposed efficient coordinate descent methods. Specifically, we consider greedy randomized coordinate descent (GRCD) \cite{BaiWu19NLAA}, greedy block coordinate descent (GBCD) \cite{LiZhang2020GBGS}, and maximal residual block Gauss-Seidel (MRBGS) \cite{LiuJiangGu2021MRBGS}.

The CT test instances are generated in the same manner as described in \cref{section 5.2}. Since projection data in practice are typically contaminated by noise, we add Gaussian noise $w(\sigma_b)\sim\mathcal{N}(0,\sigma_b^2 I_m)$ with $\sigma_b=10^{-3}$ to the right-hand side vector $b$. According to the theoretical analysis of RBSCD for noisy least-squares problems in \cref{convRBSCD-noise}, the iterates converge to the convergence horizon of the exact solution whose radius is on the order of $\|w\|_2$. Therefore, we terminate the iteration when $\mathrm{RRes}<10^{-3}$ and plot the residual error versus CPU time for each method up to this accuracy. In our implementation, RBSCD is executed using the computationally efficient method presented in \cref{eff-RBSCD}, combined with the three sampling rules introduced in \cref{section 5.2}, corresponding to $\texttt{RBSCD\_uniform}$, $\texttt{RBSCD\_normal}$, and $\texttt{RBSCD\_train}$. We emphasize that, since the candidate coordinate batches can be specified before the RBSCD iterations start, all batch-dependent quantities can be precomputed and stored, thereby significantly reducing the per-iteration cost, as demonstrated in \cref{eff-RBSCD}. By contrast, GBCD and MRBGS select their working blocks adaptively based on current residual information during the iterations, which makes such precomputation generally infeasible. \cref{fig:my_label9} compares the three RBSCD variants with GRCD, GBCD, and MRBGS in terms of residual error versus CPU time on the CT problem, where each curve reports the median over 30 runs with independent noise realizations. As a complementary measure of computational cost, \cref{fig:my_label10} reports residual error versus FLOPs for the same experiments. Each curve again reports the median over 30 runs.

\begin{figure}[tbhp]
	\centering
	\subfloat[]{\includegraphics[width=0.32\linewidth]{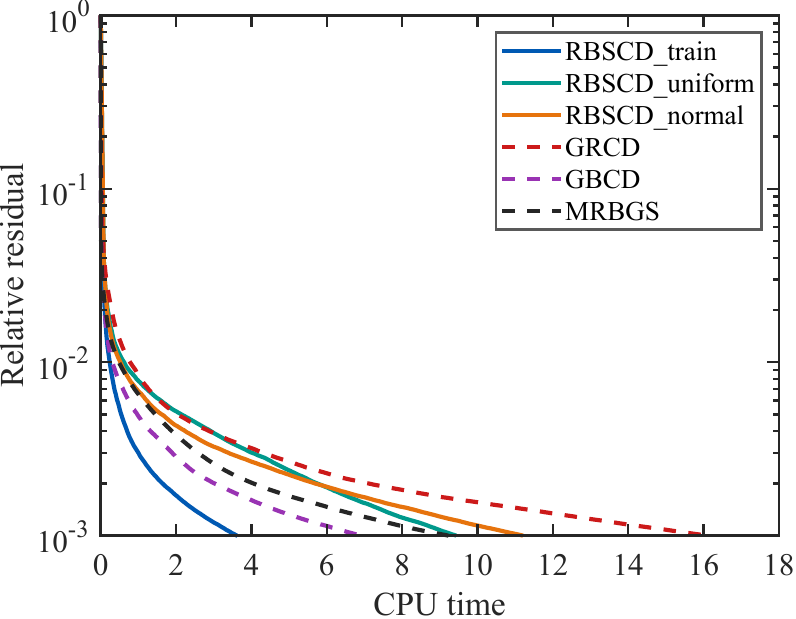}}
	\subfloat[]{\includegraphics[width=0.32\linewidth]{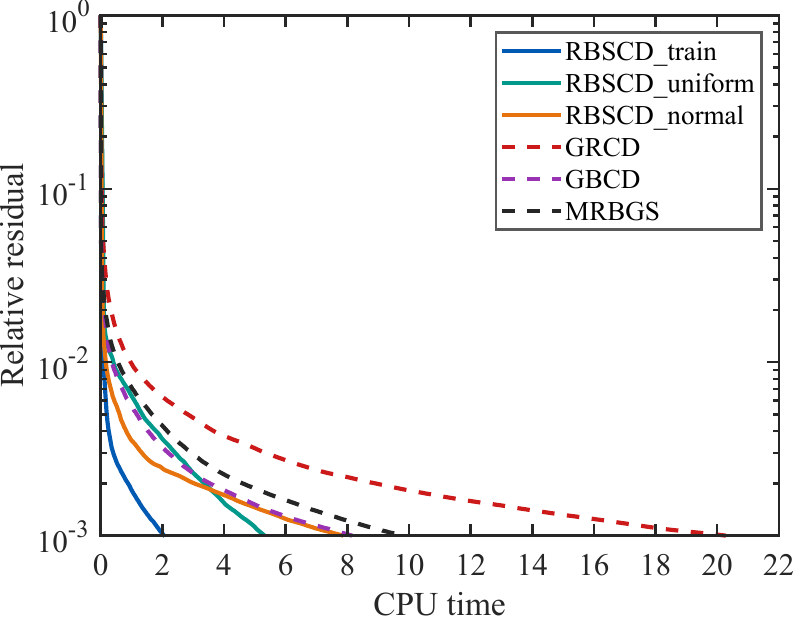}}
	\subfloat[]{\includegraphics[width=0.32\linewidth]{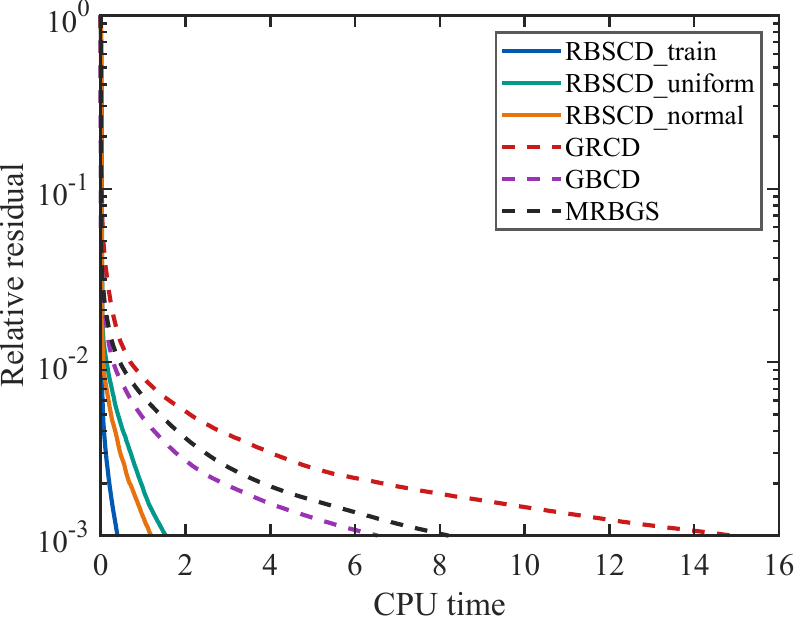}}
	\caption{Residual error versus CPU time $(\mathrm{s})$ for RBSCD and other coordinate-descent methods. Subfigures (a), (b) and (c) correspond to block sizes $q=10$, $q=20$ and $q=50$ respectively.}
	\label{fig:my_label9}
\end{figure}

\begin{figure}[tbhp]
	\centering
	\subfloat[]{\includegraphics[width=0.32\linewidth]{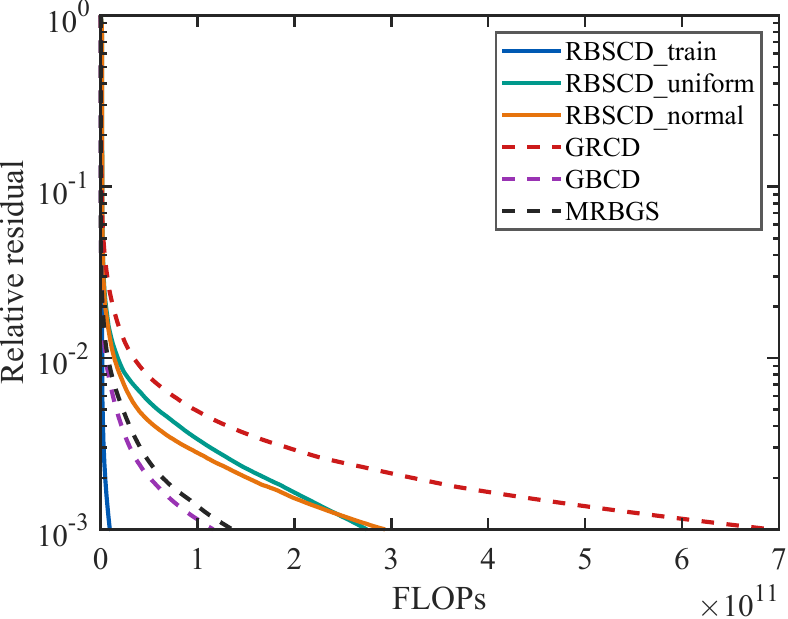}}
	\subfloat[]{\includegraphics[width=0.32\linewidth]{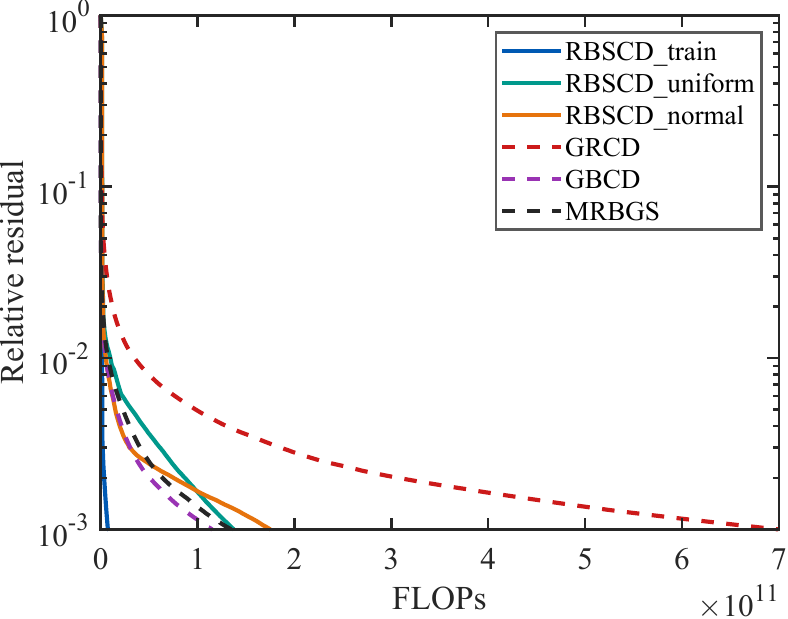}}
	\subfloat[]{\includegraphics[width=0.32\linewidth]{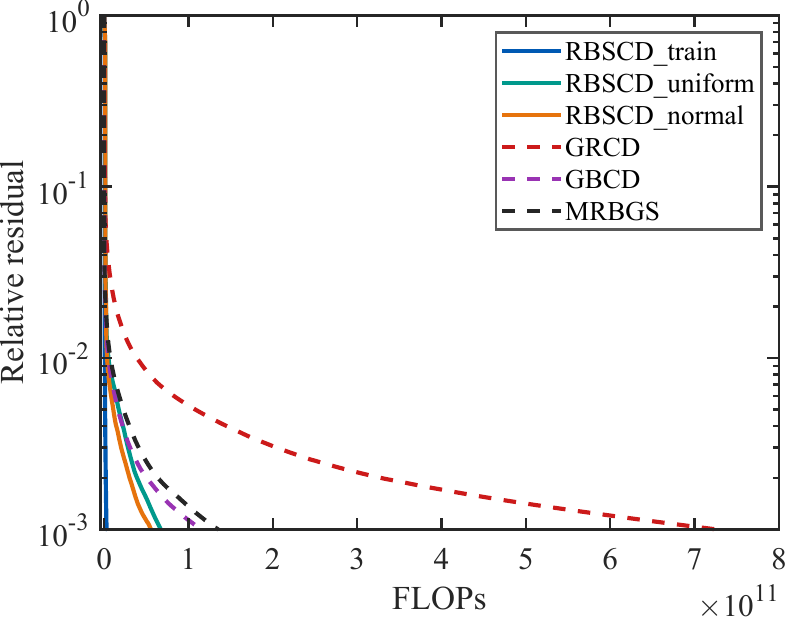}}
	\caption{Residual error versus FLOPs for RBSCD and other coordinate-descent methods. Subfigures (a), (b) and (c) correspond to block sizes $q=10$, $q=20$ and $q=50$ respectively.}
	\label{fig:my_label10}
\end{figure}

\cref{fig:my_label9} indicates that, for all tested block sizes $q$, the RBSCD variant employing the learned sampling distribution achieves the highest computational efficiency among the compared methods. In particular, for the same residual accuracy, it consistently requires significantly less CPU time than the other approaches. \cref{fig:my_label9} also supports the effectiveness of efficient implementation in \cref{eff-RBSCD}. As $q$ increases, the benefit of precomputing and storing batch-dependent quantities becomes more pronounced, leading to improved CPU-time performance. For the CT test instances considered here, when $q=50$, all three efficiently implemented RBSCD variants attain the target residual accuracy with markedly less CPU time than GRCD, GBCD, and MRBGS. This advantage is primarily attributed to the precomputation-based implementation of RBSCD, while other factors also contribute to the gap. First, greedy methods determine the working block adaptively from the current residual, which typically leads to noncontiguous memory accesses and additional computation beyond the core update. Second, in greedy block methods, the size of the selected working blocks at each iteration may be different and is difficult to control, which can result in iterations that require solving substantially larger least-squares subproblems. This leads to an unpredictable per-iteration cost and a less stable overall runtime. \cref{fig:my_label10} provides a complementary comparison in terms of arithmetic workload. For all three block sizes, $\texttt{RBSCD\_train}$ reaches the target residual using the fewest FLOPs among the compared methods. In conclusion, these results highlight the practical value of optimizing the sampling strategy in block coordinate descent and demonstrate that a learned sampling distribution can be incorporated into RBSCD to accelerate the solution of practical problems.

\subsection{Comparison with full gradient descent (GD) iteration}

To provide a comprehensive numerical evaluation, in this subsection, we compare RBSCD with the following GD iteration
\begin{equation*}
    x^{(k+1)} = x^{(k)} + \frac{A^\top (b - Ax^{(k)})}{\|A\|_2^2},
\end{equation*}
noting that GD incurs a per-iteration overhead comparable to, or even lower than, that of greedy methods. Specifically, each GD iteration requires computing the full gradient $A^\top (Ax^{(k)} - b)$, necessitating a complete pass through matrix $A$. For a dense matrix, this involves a computational workload of $\mathcal{O}(mn)$. Such per-iteration costs can become prohibitively high for large-scale problems, particularly when $A$ cannot be stored entirely in memory. In contrast, each RBSCD iteration accesses only the columns within the sampled block. Avoiding repeated full-data passes is a primary motivation for column-action methods, such as coordinate descent and block coordinate descent.

As shown in \cref{fig:supp-gd-iteration}, for the CT test problem in Section~5.3 ($n = 400$), GD requires fewer iterations than RBSCD when $q = 10$ or $q = 20$. This is attributable to the fact that each GD iteration utilizes all 400 columns of $A$, whereas each RBSCD iteration uses only 10 or 20. However, when the block size increases to $q = 50$, RBSCD train requires fewer iterations than GD, despite accessing only $50/400 = 12.5\%$ of the columns of $A$ per iteration.

The CPU time results in \cref{fig:supp-gd-time} warrant a separate interpretation. Due to hardware limitations---specifically, the memory capacity required to train a large number of distribution parameters---the test problem considered here is of moderate size. At this scale, full-matrix operations can be executed efficiently via highly optimized parallel numerical kernels, allowing GD to leverage modern computing architectures effectively. This partly explains GD's CPU time advantage when RBSCD uses relatively small block sizes (e.g., $q = 10$). However, when $q = 20$, the CPU times of GD and \(\mathrm{RBSCD}_{\mathrm{train}}\) become comparable, with the latter being slightly faster. Indeed, the ability to exploit modern parallel architectures through block operations while avoiding full-data updates is a key motivation for studying block methods like RBSCD.

For genuinely large-scale problems, GD's advantage is expected to diminish, as each iteration necessitates a complete pass through $A$, incurring substantial arithmetic, memory-access, and data-transfer costs. When $A$ cannot be stored entirely in memory, a conventional full-data GD update even becomes infeasible. In contrast, each RBSCD iteration accesses only the columns in the sampled block. Notably, even for the current moderate-sized problem, \(\mathrm{RBSCD}_{\mathrm{train}}\) outperforms GD in terms of both iteration count and CPU time when $q = 50$; see \cref{fig:supp-gd-iteration,fig:supp-gd-time}.

Collectively, these results demonstrate that an appropriately chosen batch size, combined with the learned sampling distribution, achieves a favorable balance between convergence rate and per-iteration cost. In particular, \(\mathrm{RBSCD}_{\mathrm{train}}\) achieves higher computational efficiency than full-data GD while accessing only a small fraction of the matrix columns per-iteration, further highlighting the potential of the proposed learning and reuse sampling strategy.

%%%

\begin{figure}[htbp]
	\centering
	\subfloat[]{\includegraphics[width=0.32\linewidth]{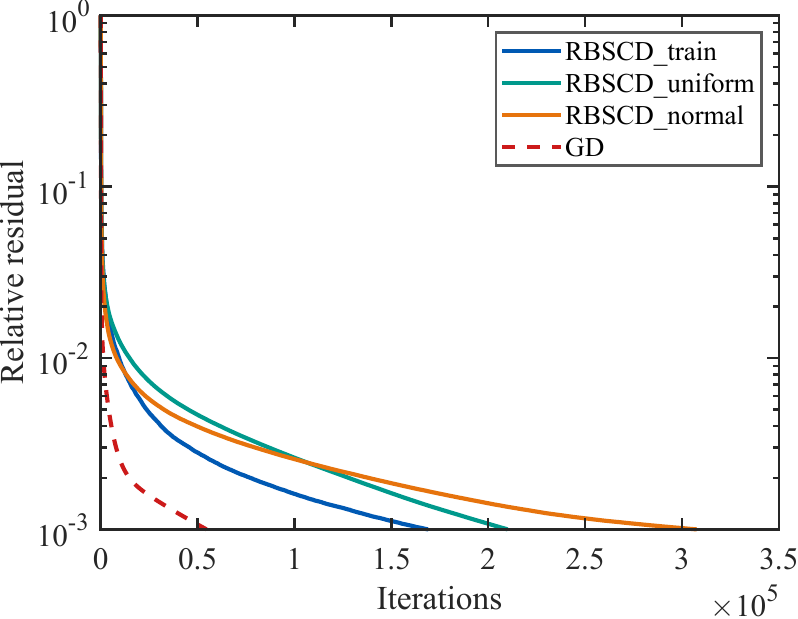}}
	\subfloat[]{\includegraphics[width=0.32\linewidth]{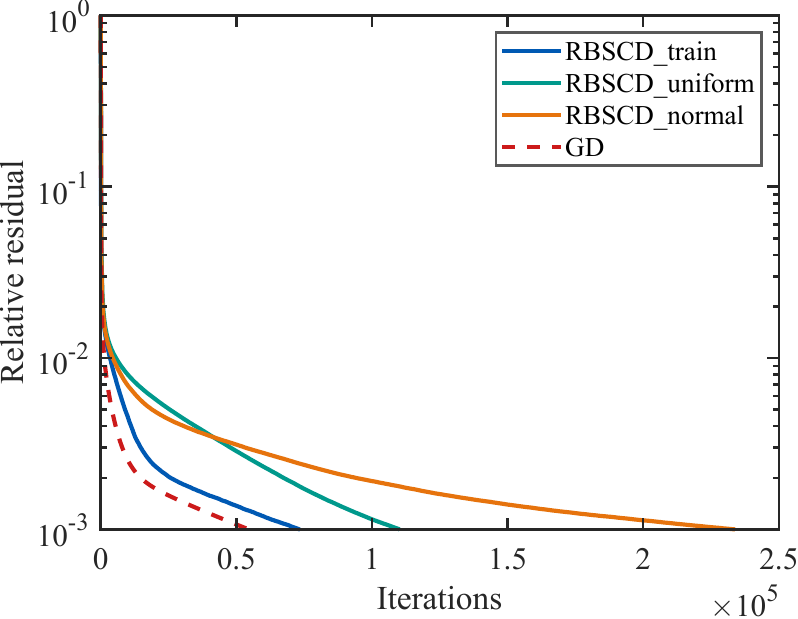}}
	\subfloat[]{\includegraphics[width=0.32\linewidth]{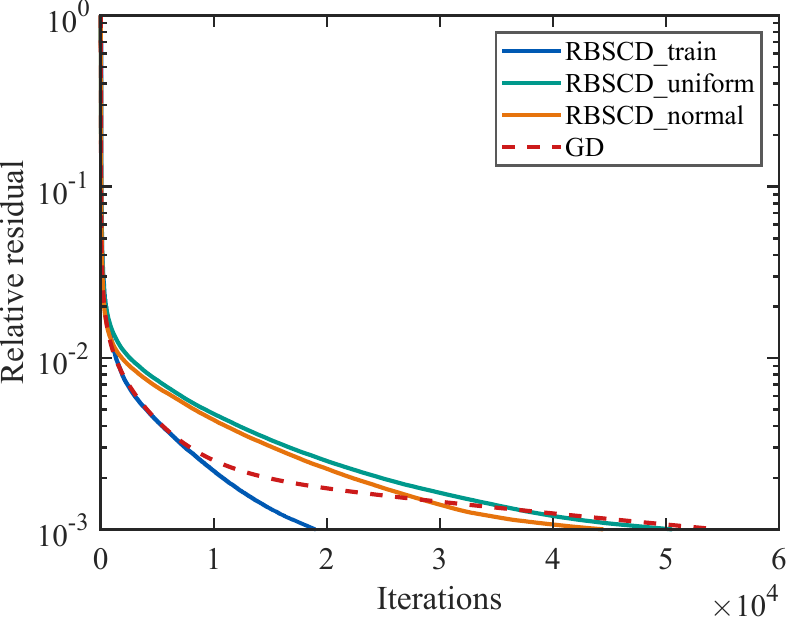}}
%	\begin{subfigure}[t]{0.32\textwidth}
%		\centering
%		\includegraphics[width=\linewidth]
%		{image/GD_iteration_q10.pdf}
%		\caption{$q=10$}
%	\end{subfigure}
%	\begin{subfigure}[t]{0.32\textwidth}
%		\centering
%		\includegraphics[width=\linewidth]
%		{image/GD_iteration_q20.pdf}
%		\caption{$q=20$}
%	\end{subfigure}
%	\begin{subfigure}[t]{0.32\textwidth}
%		\centering
%		\includegraphics[width=\linewidth]
%		{image/GD_iteration_q50.pdf}
%		\caption{$q=50$}
%	\end{subfigure}
	\caption{Residual error versus iteration count for RBSCD and GD. Subfigures (a), (b) and (c) correspond to block sizes $q=10$, $q=20$ and $q=50$ respectively.}
	\label{fig:supp-gd-iteration}
\end{figure}

\begin{figure}[htbp]
	\centering
	\subfloat[]{\includegraphics[width=0.32\linewidth]{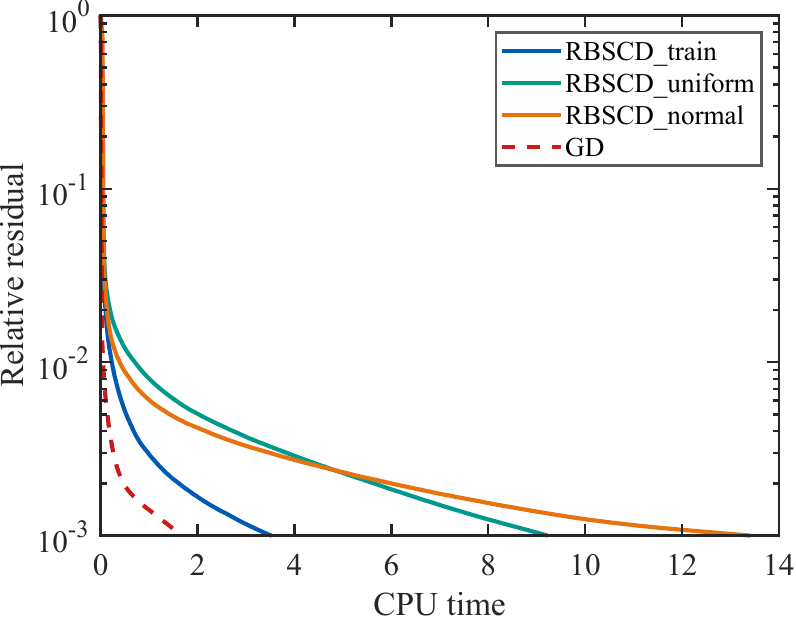}}
	\subfloat[]{\includegraphics[width=0.32\linewidth]{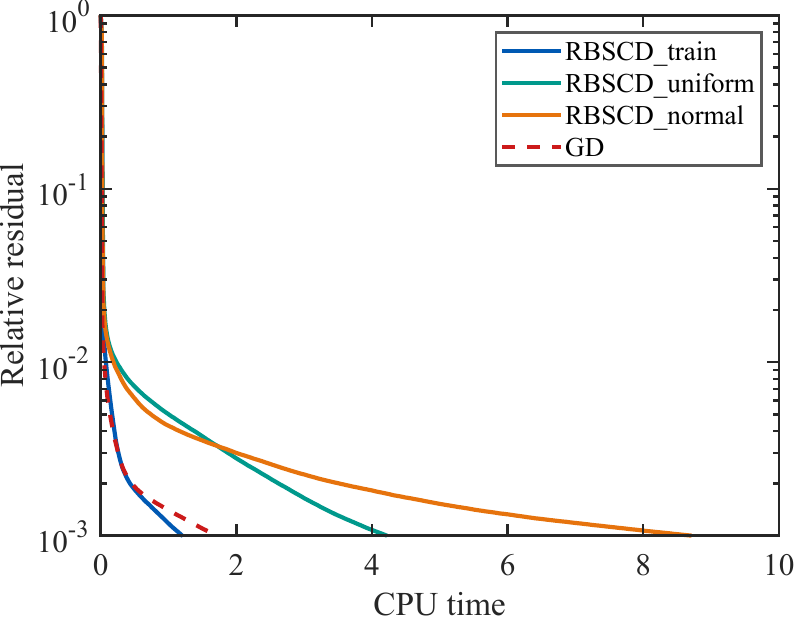}}
	\subfloat[]{\includegraphics[width=0.32\linewidth]{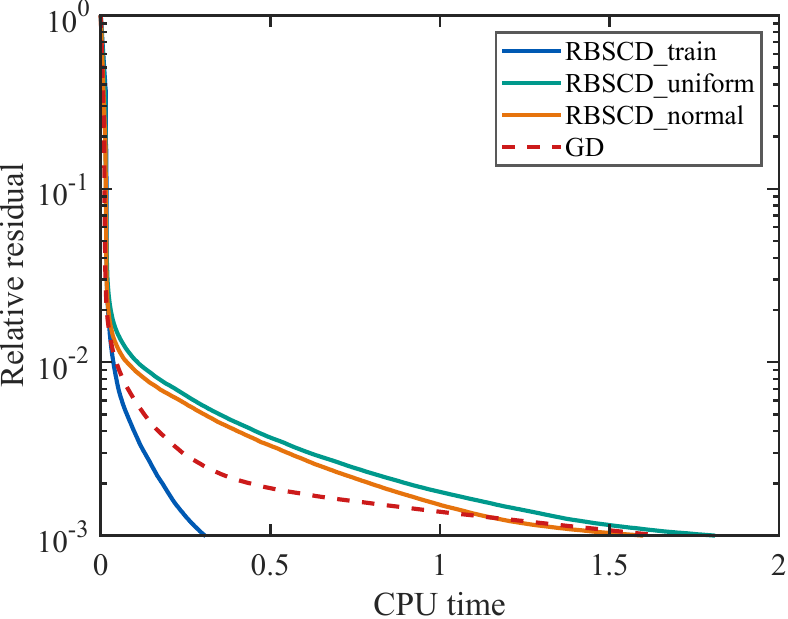}}
%	\begin{subfigure}[t]{0.32\textwidth}
%		\centering
%		\includegraphics[width=\linewidth]
%		{image/GD_time_q10.pdf}
%		\caption{$q=10$}
%	\end{subfigure}
%	\begin{subfigure}[t]{0.32\textwidth}
%		\centering
%		\includegraphics[width=\linewidth]
%		{image/GD_time_q20.pdf}
%		\caption{$q=20$}
%	\end{subfigure}
%	\begin{subfigure}[t]{0.32\textwidth}
%		\centering
%		\includegraphics[width=\linewidth]
%		{image/GD_time_q50.pdf}
%		\caption{$q=50$}
%	\end{subfigure}
	\caption{Residual error versus CPU time (s) for RBSCD and GD. Subfigures (a), (b) and (c) correspond to block sizes $q=10$, $q=20$ and $q=50$ respectively.}
	\label{fig:supp-gd-time}
\end{figure}

\section{Conclusions}\label{sec-conclusions}

By introducing a fine-grained characterization of the sampling distribution, we refine the randomized block coordinate descent method and obtain the RBSCD algorithmic framework. We further propose a unified theoretical framework to analyze the convergence of randomized coordinate descent variants for solving linear least-squares problems. On one hand, our analysis applies to any block randomized coordinate descent method under an arbitrary fixed stochastic sampling rule, yielding convergence rate bounds that are invariant to the column scaling of the coefficient matrix. On the other hand, supported by the theoretical discussion in \cref{rem:compare-needell} and the numerical evidence in \cref{section 5.1}, we demonstrate that the new convergence rate bounds are tighter than existing ones and better reflect empirical convergence behavior.

We also propose an efficient implementation of RBSCD that substantially reduces the per-iteration cost through precomputation. To further accelerate convergence in practice, we develop a strategy for learning and reusing the optimal distribution within the RBSCD framework, and apply it to CT reconstruction problems. Numerical results show that the RBSCD variant with the learned optimal distribution converges significantly faster than variants with other fixed sampling distributions. Moreover, on the test CT instances, the RBSCD method with the reused learned distribution achieves higher computational efficiency than greedy/adaptive coordinate descent methods and full gradient descent. These results indicate that the learned sampling distribution effectively captures application-specific structural features, thereby accelerating RBSCD convergence in practice.

Due to hardware limitations, we train the sampling distribution only for moderately sized CT problems. For larger-scale CT instances and other large-scale least-squares applications with pronounced structural features, designing more expressive learning models to obtain higher-quality sampling distributions represents a valuable direction for future research. Additionally, learning an optimal or near-optimal batch size for the RBSCD method within a given application class constitutes another promising avenue for future work. Finally, the aforementioned fine-grained coordinate-level analysis suggests potential for generalization to variants incorporating extension or regularization techniques for rank-deficient problems.

\section*{Acknowledgments}
The corresponding author is supported by the National Natural Science Foundation of China (Grant Nos. 12071215, 11101213).

\section*{Data Availability}
Data sharing is not applicable to this article as no new data were created or analyzed in this study.
%The data that support the findings of this study are available upon reasonable request from the authors.

\section*{Declarations}
\textbf{Conflict of interest} The authors declare that they have no conflict of interest.

%\clearpage

%\clearpage

\end{document}